\documentclass[12pt]{iopart}
\usepackage{amsthm} 
\usepackage{iopams}
\usepackage{graphicx}

\def\R{I\!\!R}
\def\N{I\!\!N}
\def\cD{\mathcal{D}}
\def\cN{\mathcal{N}}

\def\zdel{z^\delta}
\def\VXtil{{\widetilde{VX}}}
\def\VXhat{{\widehat{VX}}}
\def\Vtil{{\tilde{V}}}
\def\Vhat{{\hat{V}}}
\def\sqD{\sqrt{|D|}}
\def\cmega{{\Omega\setminus\omega}}

\newtheorem{theorem}{Theorem}[section]
\newtheorem{proposition}[theorem]{Proposition}

\newtheorem{lemma}[theorem]{Lemma}
\newtheorem{assumption}[theorem]{Assumption}

\newtheorem{remark}[theorem]{Remark}

\begin{document}
\title[An online parameter identification method for time dependent PDEs]{An online parameter identification method for time dependent partial differential equations}
\author{R Boiger$^1$, B Kaltenbacher $^1$}
\address{$^1$ Department of Mathematics, Alpen-Adria-Universit\"at Klagenfurt, Austria}
\ead{romana.boiger@aau.at and barbara.kaltenbacher@aau.at}
\begin{abstract}
Online parameter identification is of importance, e.g., for model predictive control. Since the parameters have to be identified simultaneously to the process of the modeled system, dynamical update laws are used for state and parameter estimates. 
Most of the existing methods for infinite dimensional systems either impose strong assumptions on the model or cannot handle partial observations. Therefore we propose and analyze an online parameter identification method that is less restrictive concerning the underlying model and allows for partial observations and noisy data.
The performance of our approach is illustrated by some numerical experiments.
\end{abstract}
\submitto{\IP}
\maketitle
\section{Introduction}

Dynamical systems like ordinary differential equations or time-dependent partial differential equations play an important role for modeling instationary processes in science and technology. Such models often contain parameters that cannot be accessed directly and therefore must be determined from measurements, which leads to inverse problems. In many applications, e.g., in model predictive control, the parameter identification has to take place during the operation of the considered system. Hence online methods become necessary. 
Examples of applications range from HVAC (heating ventilation airconditioning) systems via battery charge estimation to aircraft dynamics, see e.g. \cite{Jategaonkar06}, \cite{Maasoumy14}, \cite{Rahimi14}.

In many applications we face the additional problem of having only partial and noisy observations of the state. 
Motivated by these facts, in this paper we propose an online identification method that is also applicable in case of indirect partial observations and takes into account noisy data. For this purpose we employ a dynamic update law for both the estimated parameters and the state estimate that is strongly inspired by the schemes from \cite{Baumeister97} and \cite{Kuegler10}.
Online parameter identification has been extensively studied in the finite dimensional setting, e.g. \cite{Ioannou}, \cite{Narendra05} or \cite{Sastry}.
The literature becomes much more scarce when dealing with infinite dimensional models as arising in the context of partial differential equations. We refer to the extensive literature review in \cite{Baumeister97} and \cite{Kuegler10}. More recent work on this topic can e.g. be found in \cite{DihlmannHaasdonk}.

The paper is organized as follows:
In section 2 we state the underlying differential equation with the according assumptions and define the online parameter identification method. In the next section the convergence analysis of the method is discussed for the exact data case, the case with noisy data and also the one with smooth noisy data.
Some examples and numerical experiments illustrate the performance of the method in section 4. We conclude with some remarks and an outlook in section 5. 

\section{Online Parameter Identification method}
In this chapter we present the underlying differential equation and the corresponding assumptions. Further we introduce an online parameter identification method. 

Let $Q$, $X$ and $Z$ be Hilbert spaces.
We consider the abstract ordinary differential equation 
\begin{eqnarray}
q_t(t,x)=0 \label{eq:main}\\
u_t(t,x)+C(q(t),u(t))(x)=f(t,x) \nonumber\\
u(0,x)=u_0(x) \nonumber\\
q(0,x)=q_0(x) \nonumber
\end{eqnarray}
where $C\colon Q\times\cD(C) (\subseteq Q\times X)\to X$, $f\colon[0,\infty)\times X \to X$ and the initial value for $u$, namely $u_0$ are given.
The inverse problem we are interested in is to find the parameter $q$ from given observations of the state $u$ over time, $Gu(t,x)=z(t,x)$, where $G\colon X \to Z$ is the observation operator and $Z$ the observation space. For simplicity of exposition we consider a linear observation operator here. Most of what follows can be carried over to the case of nonlinear observations.

We will denote the exact solution by $q^{\ast}$ and $u^{\ast}$.
To define an evolution system for identifying $q^\ast$ from measurements $z$ we split $u^\ast$ in its ``observed part'' $Ru^\ast=G^\dagger z \in \mathcal{N}(G)^\bot\subseteq \tilde{V}$ and its ``unobserved'' part $Pu^\ast=u^\ast-Ru^\ast \in \mathcal{N}(G) \subseteq \Vhat$ by appropriate projections $R$ and $P$. Here $\Vtil\subseteq \VXtil \subseteq X$ and $\Vhat \subseteq \VXhat \subseteq X$ with the corresponding embedding constants $C_{\tilde{V}\widetilde{VX}},C_{\tilde{VX}X},C_{\hat{V}\widehat{VX}},C_{\widehat{VX}X}$ and the operator $G^\dagger \colon Z \to X$ is the Moore-Penrose Inverse of $G$. 
Hence the projection $R$ for the ``observed'' part is the projection on the orthogonal complement of the nullspace of $G$, namely $R \colon X \to \mathcal{N}(G)^\bot$, $R=G^\dagger G$.
The orthogonal projection $P$ is the projection on the nullspace of $G$, that is $P\colon X \to \mathcal{N}(G)$, $P=I-R$.
\begin{assumption}\label{ass:equation}
For the abstract ODE (\ref{eq:main}) we assume that
\begin{enumerate}
\item the exact solution $u^\ast$ exists and stays bounded, i.e. for all times $t>0$ we have $u^\ast(t) \in \mathcal{B}_\rho(u_0) \subseteq \mathcal{D}(C)$, \newline
where $\mathcal{B}_\rho(u_0)=\left\{v+w \in \tilde{V}+\hat{V} \mid \left\|v-Ru_0\right\|_{\tilde{V}}+\left\|w-Pu_0\right\|_{\hat{V}}\leq \rho \right\}$;
\item the operator $C$ satisfies a Lipschitz condition with respect to the second variable, i.e. for all times $t>0$ and for all $v+w\in \tilde{V}+\hat{V}$ 
\begin{equation}
\label{eq:LipschitzC}
\left\|C\left(q^\ast\, , u^\ast(t)+v+w\right)-C\left(q^\ast \, , u^\ast(t)\right)\right\|_X \leq L_C \left(\left\|v\right\|_{\tilde{V}}+\left\|w\right\|_{\hat{V}}\right)
\end{equation} holds;
\item the operator $C$ can be split in a part that is dependent of $q$ and the rest: 

$C(q,u)=A(u)q+B(u)$;
\item for all $u\in \mathcal{B}_\rho(u_0)$ the operator $A(u) \colon Q \to X$ is linear and bounded and there exists $C_A>0$ such that 
\begin{equation}\|A(u^\ast+v)\|_{Q \rightarrow X}\leq C_A(1+\|v\|_{\hat{V}}) \quad \forall \, v \in \hat{V} \label{AbdV}
\end {equation}
 or 
\begin{equation}
\|A(u^\ast+v)\|_{Q \rightarrow X}\leq C_A(1+\|v\|_X) \quad \forall \, v \in X \label{AbdX}
\end{equation}
or 
\begin{equation} \label{AbdVV}
\|A(u^\ast+\hat{v}+\tilde{v})\|_{Q\rightarrow X}\leq C_A(1+\|\hat{v}\|_\Vhat+\|\tilde{v}\|_\Vtil) \quad \forall \, \hat{v} \in \Vhat,\, \tilde{v}\in \Vtil;
\end{equation}
\item there exist coercive and bounded operators $M \colon \tilde{V} \to X$ and $N \colon  \hat{V} \to X$ i.e.
\begin{itemize}
\item there exist constants $c_M$ and $C_M$ such that for all $v \in \tilde{V}$ $\left(Mv,v\right)_X \geq c_M \left\|v\right\|^2_{\widetilde{VX}}$ and $\left\|RMv\right\|_X\leq C_M \left\|v\right\|_{\tilde{V}}$;
\item  there exist constants $c_N$ and $C_N$ such that for all $v \in \hat{V}$ $\left(Nv,v\right)_X \geq c_N \left\|v\right\|^2_{\widehat{VX}}$ and $\left\|PNv\right\|_X\leq C_N \left\|v\right\|_{\hat{V}}$.
\end{itemize}
\end{enumerate}
\end{assumption}
Note that by continuity of the embeddings $\Vhat \hookrightarrow X$, $\Vtil+\Vhat \hookrightarrow X$, (\ref{AbdX}) is sufficient for (\ref{AbdV}), (\ref{AbdVV}).
Conditions 1., 2. and 3. are similar to Assumptions 1 and 2 in \cite{Kuegler10}.

Now we want to introduce our online parameter identification method. Online identification means that the parameter identification, the data collection process and the operation of the system are taking place at the same time. Accurate parameter values are needed for making decisions while the system is in operation. Therefore our online parameter identification method includes a dynamical update law for the parameter and state estimates. 

\begin{eqnarray}
\hat{q}_t-A(Ru^\ast+P\hat{u})^*(R\hat{u}-Ru^\ast)=0
\label{eq:est_exact_q}\\
\hat{u}_t+C(\hat{q},Ru^\ast+P\hat{u})+\mu RM \frac{R\hat{u}-Ru^\ast}{\|R\hat{u}-Ru^\ast\|_\Vtil}+\nu PN P\hat{u} =f
\label{eq:est_exact_u}\\
(\hat{q},\hat{u})(0)=(\hat{q}_0,\hat{u}_0)
\label{eq:est_exact0}
\end{eqnarray}
where $\hat{u}_0$ need not coincide with $u_0$.

The method is strongly motivated by the methods proposed by K\"ugler \cite{Kuegler10} and by Baumeister et. al. \cite{Baumeister97}. The main difference compared to \cite{Baumeister97} is that we also allow for partial observations, which often occur in applications. This is also to some extent possible with the method from \cite{Kuegler10}, however in contrast to \cite{Kuegler10} we do not assume monotonicity of the operator $C$.
\section{Convergence Analysis}
In this chapter we consider convergence of the estimator in the exact data case as well as in case of noisy or smooth noisy data, respectively.
To do so we take a look at the errors between the exact solution $(q^\ast,Ru^\ast, Pu^\ast)$ and the estimated parameter $\hat{q}$ as well as the error in the projected states $R\hat{u}$ and $P\hat{u}$ that we denote by $e$, $r$ and $p$.
The error components 
\begin{equation}\label{erp}
e=\hat{q}-q^\ast\,, \quad r=R\hat{u}-Ru^\ast\,, \quad p=P\hat{u}-Pu^\ast 
\end{equation}
satisfy the following system of differential equations, where we split up the differential equation for the state in the ``observed'' and the ``unobserved'' part
\begin{eqnarray}
\fl e_t-A(u^\ast+p)^\ast r=0  \label{eq:errorsys_e}\\
\fl r_t+RC(q^\ast,u^\ast+p)-RC(q^\ast,u^\ast)+RA(u^\ast+p)e+\mu R M \frac{r}{\left\|r\right\|_{\tilde{V}}}=0 \label{eq:errorsys_r}\\
\fl p_t+PC(q^\ast,u^\ast+p)-PC(q^\ast,u^\ast)+PA(u^\ast+p)e+\nu PNP \hat{u}=0 \label{eq:errorsys_p}\\
\fl (e,r,p)(0)=(\hat{q}_0-q^\ast, R(\hat{u}_0-u_0), P(\hat{u}_0-u_0)). \nonumber
\end{eqnarray}
Here we have used the identities $Ru^\ast+P\hat{u}=u^\ast+p$ and
\begin{equation} \label{id1}
\fl C(\hat{q},Ru^\ast+P\hat{u})-C(q^\ast,u^\ast)\pm C(q^\ast, R u^\ast +P\hat{u})=C(q^\ast,u^\ast+p)-C(q^\ast,u^\ast)+A(u^\ast+p)e.
\end{equation}
as well as Assumption \ref{ass:equation}.

\subsection{Convergence with exact data}
\subsubsection{Well-definedness}
To obtain existence and boundedness of the solutions according to our method (\ref{eq:est_exact_q}), (\ref{eq:est_exact_u}), (\ref{eq:est_exact0}), we first multiply (\ref{eq:errorsys_e}) and (\ref{eq:errorsys_r}) with $e$ and $r$ respectively and integrate with respect to time over an interval $[t_1,t_2]$, $t_1$, $t_2>0$ to get, using Assumption \ref{ass:equation},

\begin{eqnarray}
\frac{1}{2}\left[\left\|e\right\|^2_Q+\left\|r\right\|^2_X\right]^{t_2}_{t_1}=\int^{t_2}_{t_1}\left\{(e_t,e)_Q+(r_t,r)_X\right\}d\tau \nonumber\\
=-\int^{t_2}_{t_1}\left\{\Big(RC(q^\ast,u^\ast+p)-RC(q^\ast,u^\ast)+\mu RM \frac{r}{\left\|r\right\|_{\tilde{V}}},r\Big)_X \right\}d\tau \nonumber \\
\leq - \int^{t_2}_{t_1}\left\{-L_C\left\|p\right\|_{\hat{V}}\left\|r\right\|_X+c_M \mu \frac{\|r\|^2_{\widetilde{VX}}}{\|r\|_{\tilde{V}}}\right\}d\tau. \label{align:neplusnr}
\end{eqnarray}
We see that the equation for $\hat{q}$ was designed such that the terms containing $A$ cancel out.
The above estimate leads us to choose $\mu$ according to 
\begin{assumption}
\label{cond:mu_1} For all $t>0$
$$\quad \mu(t)\geq \frac{2 L_C}{c_M} \|p(t)\|_{\Vhat} \frac{\|r(t)\|_X\|r(t)\|_{\Vtil} }{\|r(t)\|_{\VXtil}^2} \,.$$
\end{assumption}
\noindent
Therewith we obtain
\begin{eqnarray*}
\frac{1}{2} \left[\|e\|_Q^2+\|r\|_X^2\right]_{t_1}^{t_2}\leq-\int^{t_2}_{t_1}\left\{-L_C\|p\|_{\hat{V}}\|r\|_X+2L_C\|p\|_{\hat{V}}\|r\|_X\right\}d\tau \\
 \leq -L_C\int_{t_1}^{t_2} \|p\|_{\Vhat} \|r\|_X \, d \tau <0 \,.
\end{eqnarray*}
This particularly implies boundedness
$$
\forall \, t>0\, : \ \|e(t)\|_Q^2+\|r(t)\|_X^2\leq\|e(0)\|_Q^2+\|r(0)\|_X^2\,,
$$
and finiteness of the integral
$$
\forall \, T>0\, : \ \int_0^T \|p\|_{\Vhat} \|r\|_X \, dt \leq \frac{\|e(0)\|_Q^2+\|r(0)\|_X^2}{2L_C}<\infty\,.
$$

Now it remains to find an appropriate bound for the error of the ``unobserved'' part of the state, which can be done quite similarly. For this purpose we multiply (\ref{eq:errorsys_p}) with $p$ and use Assumption \ref{ass:equation} with (\ref{AbdV}) as well as (\ref{id1}) to gain 
\begin{eqnarray*}
\fl \frac{d}{dt}\frac{1}{2} \left[\| p \|^2_{X}\right]= (p_t,p)_X\\
\fl = -(PC(q^\ast,u^\ast+p)-PC(q^\ast,u^\ast),p)_X+(PA(u^\ast+p)e,p)_X - \left(\nu PNP \hat{u},p\right)_X\\
\fl \leq \|C(q^\ast,u^\ast)-C(q^\ast,u^\ast+p)\|_X \|p\|_X+\|A(u^\ast+p)\|_{Q\rightarrow X}\|e\|_Q\|p\|_X-\nu (PNP\hat{u},p)_X \\
\fl \leq L_C\|p\|_{\hat{V}}\|p\|_X+C_A(1+\|p\|_{\hat{V}})\|e\|_Q\|p\|_X -\nu (PN(p+Pu^\ast),p)_X 
\end{eqnarray*}
For the second and the last term we use Assumption \ref{ass:equation}, the embedding inequalities and Young's inequality to get 
$$
C_A\|e\|_Q\|p\|_X\leq \frac{C_A}{2}\left[\|e\|^2_Q+\|p\|^2_X\right] \leq \frac{C_A}{2}\|e\|^2_Q+\frac{C_A}{2}C_{\widehat{VX}X}C_{\hat{V}\widehat{VX}}\|p\|_X\|p\|_{\hat{V}},
$$
$$
-\nu (PNp,p)_X \leq -\nu c_N \|p\|^2_{\widehat{VX}}
$$
and 
\begin{equation}\label{est1}
\fl -\nu(PNPu^\ast,p)_X \leq \nu C_N \|Pu^\ast\|_\Vhat \|p\|_X \leq \nu \left(\frac{C_N^2 C^2_{\VXhat X}}{2 c_N}\|Pu^\ast\|^2_{\hat{V}}+\frac{c_N}{2}\|p\|^2_{\widehat{VX}}\right).
\end{equation}
So altogether we have
\begin{eqnarray}\label{estp}
\frac{d}{dt}\frac{1}{2} \left[\| p \|^2_{X}\right] 
\leq (L_C+C_A(\|e\|_Q+\frac{1}{2}C_{\widehat{VX}X}C_{\hat{V}\widehat{VX}}))\|p\|_X\|p\|_{\hat{V}} \\+\frac{C_A}{2}\|e\|^2_Q-\nu \frac{c_N}{2}\|p\|^2_{\widehat{VX}}+\nu \frac{C_N^2 C^2_{\widehat{VX}X}}{2 c_N}\|Pu^\ast\|^2_{\hat{V}}. \nonumber
\end{eqnarray}
This leads us to choose $\nu$ according to 
\begin{assumption}\label{cond:nu}
$$
\forall t>0 \ : \quad \nu(t)\geq \max\left\{\underline{\nu}\,,\ \frac{4(L_C+C_A(\|e(t)\|_Q+\frac{1}{2} C_{\Vhat\VXhat}C_{\VXhat X}))}{c_N} \frac{\|p(t)\|_\Vhat\|p(t)\|_X}{\|p(t)\|_\VXhat^2}\right\} 
$$
\end{assumption}
\noindent
to obtain
$$
\frac{d}{dt}\frac{1}{2} \left[\|p\|_X^2\right] 
\leq-\nu \frac{c_N}{4}\|p\|_{\widehat{VX}}^2 
+\nu \frac{C_N^2 C_{\widehat{VX} X}^2}{2c_N}\|Pu^\ast\|_{\hat{V}}^2+\frac{C_A}{2}\|e\|_Q^2.
$$
We now define $\tilde{\mathcal{V}}(\tau(t)):=\mathcal{V}(t)=\frac{1}{2}[\|p(t)\|^2_X]$ and $\tau(t):=\frac{c_N C^2_{\widehat{VX}X}}{2}\int^t_0 \nu (\xi) d\xi$ and hence $\frac{d\tau}{dt}=\frac{c_N C^2_{\widehat{VX}X}}{2}\nu (t)$.
Using the former estimate we get 
\begin{eqnarray*}
\frac{d}{d\tau}\tilde{\mathcal{V}}(\tau(t))=\frac{d}{dt}\mathcal{V}(t)\frac{1}{\frac{d\tau}{dt}}=\frac{d}{dt}\frac{1}{2}[\|p(t)\|^2_X]\frac{2}{c_N C^2_{\widehat{VX}X}\nu(t)} \\
 \leq -\frac{1}{2}\|p\|^2_{\widehat{VX}}\frac{1}{C^2_{\widehat{VX}X}}+\frac{C_N^2}{c_N^2}\|Pu^\ast\|^2_{\hat{V}}+\frac{C_A}{c_N C^2_{\widehat{VX}X}\nu}\|e\|^2_Q \\
 \leq -\tilde{\mathcal{V}}(\tau) +\frac{C_N^2}{c_N^2}\sup_{t>0}{\|Pu^\ast(t)\|^2_{\hat{V}}}+\frac{C_A}{c_N C^2_{\widehat{VX}X}\underline{\nu}}\sup_{t>0}{\|e(t)\|^2_Q}.
\end{eqnarray*}
Here we use the fact that for any differentiable nonnegative function $\eta : [0,T]\rightarrow \R^+_0$ and $a,b >0$ and for all $t\in [0,T]$ the following implication holds:
$$ \eta^{'}(t)\leq -a\eta(t)+b \Rightarrow \eta(t)\leq \frac{b}{a}+(\eta(0)-\frac{b}{a})e^{-at} \leq \max \left\{\frac{b}{a},\eta(0)\right\}.
$$
So with $a=1$ and $b=\frac{C_A}{c_N C^2_{\widehat{VX}X} \underline{\nu}}\sup_{t>0}{\|e(t)\|^2_Q}+\frac{C_N^2}{c_N^2}\sup_{t>0}{\|Pu^\ast (t)\|^2_{\hat{V}}}$ we get:
\begin{proposition}
\label{prop:welldefine}
Let Assumptions \ref{ass:equation} with (\ref{AbdV}), \ref{cond:mu_1}, and \ref{cond:nu} hold and let $(\hat{q}_0-q^\ast,\hat{u}_0-u_0)\in Q \times (\tilde{V}+\hat{V})$. Then there exists a solution $(\hat{q}(t),\hat{u}(t))\in Q \times (\tilde{V}+\hat{V})$ for all $t>0$ and the following estimates on the parameter and state errors (cf. (\ref{erp})) hold.
\begin{enumerate}
\item For all $t>0$: $\|e(t)\|^2_Q +\|r(t)\|^2_X \leq \|e(0)\|^2_Q +\|r(0)\|^2_X$;
\item For all $t>0$: 
$\|p(t)\|_{X} \leq \max\left\{\|p(0)\|_{X},\frac{C_A^2}{c_N C^2_{\widehat{VX}X} \underline{\nu}}\left(\|e(0)\|^2_Q +\|r(0)\|^2_X\right)+\frac{C_N^2}{c_N^2}\sup_{t>0}{\|Pu^\ast (t)\|^2_{\hat{V}}} \right\}$;
\item $\int^\infty_0 \|p(t)\|_{\hat{V}}\|r(t)\|_X dt \leq \frac{ \|e(0)\|^2_Q +\|r(0)\|^2_X}{L_C} < \infty$.
\end{enumerate}
\end{proposition}
\subsubsection{State convergence}
In this section we will show that the estimated ``observed'' state converges towards the ``observed'' part of the exact solution.
For improving the state convergence we impose an additional lower bound on $\mu$ as compared to Assumption \ref{cond:mu_1} (note that therewith Proposition \ref{prop:welldefine} still remains valid). 
\begin{assumption}\label{cond:mu02}
There exists a constant $c_1>0$ such that for all $t>0$
$$
\mu(t)\geq \max\left\{
\frac{2L_C}{c_M} \|p(t)\|_{\Vhat}\,, \ c_1 \|r(t)\|_X\right\}
\frac{\|r(t)\|_X\|r(t)\|_{\Vtil} }{\|r(t)\|_{\VXtil}^2}.
$$
\end{assumption}
\begin{theorem}[State convergence]
\label{th:state}
Under Assumptions \ref{ass:equation} with (\ref{AbdX}), \ref{cond:nu}, and \ref{cond:mu02}
we have that $\left\|R(\hat{u}(t)-u^*(t))\right\|_X=\left\|r(t)\right\|_X\rightarrow 0$ as $t \rightarrow \infty$.
\end{theorem}
\begin{proof}
We first take a look at the ``observed'' state error for $t_2>t_1>0$, for which we get from (\ref{eq:errorsys_r}) and (\ref{id1})  
\begin{eqnarray*}
\fl \|r(t_2)\|^2_X-\|r(t_1)\|^2_X =\int^{t_2}_{t_1}{\frac{d}{dt}\|r(t)\|^2_X}dt = \int^{t_2}_{t_1}{(r_t, r)_X}dt \\
\fl =\int^{t_2}_{t_1}{ \underbrace{\Big(R\left(C(q^\ast,u^\ast)-C(q^\ast,u^\ast+p)\Big),r\right)_X}_{(1)}-\underbrace{\Big(RA(u^\ast+p)e,r\Big)_X}_{(2)}-\underbrace{\left(\mu RM\frac{r}{\|r\|_{\tilde{V}}},r\right)_X}_{(3)}}dt
\end{eqnarray*}
where we have to estimate these terms appropriately.
By Assumption \ref{ass:equation} the second term $(2)$ can be estimated by
$$
|(R(C(q^\ast,u^\ast)-C(q^\ast,u^\ast+p)),r)_X|   \leq L_C \|p\|_{\hat{V}}\|r\|_X.
$$
Similarly for term $(3)$ we have with Assumption (\ref{ass:equation}) with (\ref{AbdX})
\begin{eqnarray*}
\fl |(RA(u^\ast+p)e,r)_X|\leq \left\|A(u^\ast+p)\right\|_{Q\rightarrow X}\|e\|_Q\left\|r\right\|_X \leq C_A(1+\sup_{t>0}\|p(t)\|_X)\|e\|_Q\|r\|_X \\
\fl \leq \frac{L_A}{2}(\|e\|^2_Q+\|r\|^2_X)
\end{eqnarray*}
with \begin{equation}L_A:=C_A(1+\sup_{t>0}\|p(t)\|_X), \label{def:LA}\end{equation}
which is finite by Proposition \ref{prop:welldefine}.
Using Assumptions \ref{ass:equation} and \ref{cond:mu02} we get for term $(1)$
$$
-\left(\mu RM\frac{r}{\|r\|_{\tilde{V}}},r\right)_X \leq -\frac{\mu}{\|r\|_{\tilde{V}}}c_M \|r\|^2_{\widetilde{VX}} \leq -2 L_C \|p\|_{\hat{V}}\|r\|_X.
$$
Altogether we have
\begin{eqnarray*}
\fl \|r(t_2)\|_X^2-\|r(t_1)\|_X^2  \leq \int^{t_2}_{t_1}{\left\{-L_C\|p(t)\|_{\hat{V}}\|r(t)\|_X+\frac{L_A}{2}(\|e(t)\|^2_Q+\|r(t)\|^2_X)\right\}}dt \\
\fl \leq \int^{t_2}_{t_1}{\frac{L_A}{2}(\|e(t)\|^2_Q+\|r(t)\|^2_X)}dt  \leq c_2(t_2-t_1)
\end{eqnarray*}
with $c_2:=\frac{L_A}{2}(\|e(0)\|_Q^2+\|r(0)\|^2_X)$, which follows from Proposition \ref{prop:welldefine}.
Using this estimate we get for any $t$, $\gamma>0$ fixed
\begin{eqnarray*}
\fl \gamma \|r(t)\|^2_X =\int^t_{t-\gamma}\{\|r(\tau)\|^2_X+(\|r(t)\|^2_X - \|r(\tau)\|^2_X )\}d\tau \\ \fl \leq\int^t_{t-\gamma}{\|r(\tau)\|^2_X}d\tau +c_2\int^t_{t-\gamma}{(t-\tau)}d\tau 
 =\int^t_{t-\gamma}{\|r(\tau)\|^2_X}d\tau+c_2\frac{\gamma^2}{2} \,.
\end{eqnarray*}
Hence we have for all $t,\gamma >0$ that
\begin{equation}
\label{ineq:gamma}
\int^{t}_{t-\gamma}\|r(\tau)\|^2_X \geq \gamma \|r(t)\|^2_X-\frac{c_2 \gamma^2}{2} \,.
\end{equation}
From (\ref{align:neplusnr}) and choosing $\mu$ according to Assumption \ref{cond:mu02} we get
\begin{eqnarray}
\fl \frac{1}{2} \left[\|e(t)\|_Q^2+\|r(t)\|_X^2\right]^{t_2}_{t_1} \leq  -\int^{t_2}_{t_1}{\left\{-L_C \|p(t)\|_{\hat{V}}\|r(t)\|_X+c_M \mu \frac{\|r(t)\|^2_{\widetilde{VX}}}{\|r(t)\|_{\tilde{V}}}\right\}}dt \nonumber \\
\fl \leq -\int^{t_2}_{t_1} \mu \frac{c_M}{2} \frac{\|r(t)\|_{\VXtil}^2}{\|r(t)\|_{\Vtil} }  dt \leq -\frac{c_M c_1}{2}\int_{t_1}^{t_2} \|r(t)\|_X^2   dt\,, \label{eq:est_er1}
\end{eqnarray}
hence
\begin{equation}\label{eq:finiteintegral_r2}
\int_0^\infty \|r\|_X^2 \, dt \leq \frac{\|e(0)\|_Q^2+\|r(0)\|_X^2}{c_Mc_1}<\infty\,.
\end{equation}

We want to show that $\lim_{t\rightarrow \infty}{\|r(t)\|}=0$.
So we suppose that $\lim_{t\rightarrow \infty}{\|r(t)\|}\neq 0$. If this is the case then there exists a sequence $(t_i)_{i \in \N}$ with $t_i \rightarrow \infty$ for $i\rightarrow \infty$, and 
 an $\varepsilon >0$ such that for all $i\in\N$ $\|r(t_{i})\|^2_X \geq \varepsilon$.
Now we select a subsequence $(t_{i_{j}})_{j\in\N}$ such that for all $j\in\N$ we additionally have $t_{i_{j}}-t_{i_{j-1}} \geq \frac{\varepsilon}{c_2}$.
Because of inequality (\ref{ineq:gamma}), choosing $\gamma=\frac{\varepsilon}{c_2}$
 we have
$$
\frac{\varepsilon^2}{2c_2} \leq \int^{t_{i_{j}}}_{t_{i_{j}}-\gamma}\|r(\tau)\|^2_X d\tau 
$$
By summing up on both sides and using $t_{i_{j}}-\gamma\geq t_{i_{j}}-\frac{\varepsilon}{c_2}\geq t_{i_{j-1}}$ we get for all $n \in \N$
$$
n \, \frac{\varepsilon^2}{2c_2} \leq 
\sum^{n}_{j=1} \int^{t_{i_{j}}}_{t_{i_{j}}-\gamma} \|r(\tau)\|_X^2 d\tau 
 \leq  \int^{t_{i_{n}}}_0 \|r(\tau)\|_X^2 d\tau \\
 \leq \int^\infty_0 \|r(\tau)\|_X^2 d\tau\,,
$$
which gives a contradiction to (\ref{eq:finiteintegral_r2}).
\end{proof}
\subsubsection{Parameter convergence}
The proofs in this section are to some extent similar to those in Section 3 of \cite{Kuegler10}. Note however, that the Lemma quantifying the relation between state error and parameter error can be stated in a stronger manner (cf. Lemma \ref{lemma1} below), which enables to considerably simplify the final convergence proof, see Theorem \ref{th:parameter} below. 
In order to show that the parameter error converges to zero we start with some preparatory results. First we prove an estimate on the norm of the ``observed'' state error.
\begin{lemma}\label{lem:est_state}
Under Assumption \ref{ass:equation} with (\ref{AbdX}), the projected state errors $r=R(\hat{u}-u^*)$ and $p=P(\hat{u}-u^*)$ satisfy the following relation for all $0 \, <t_a\leq t_b\leq t_c$
\begin{eqnarray*}
\fl \|r(t_c)\|_X \geq
\|\int^{t_c}_{t_b} RA(u^\ast(\tau)+p(\tau))e(t_a)d\tau \|_X\\
\fl -\|r(t_b)\|_X
-L_A^2\int_{t_b}^{t_c}\left\{\int_{t_a}^\tau \|r(\sigma)\|_X\, d\sigma\right\}d\tau 
-L_C \int_{t_b}^{t_c}\|p(\tau)\|_\Vhat\, d\tau
-C_M \int_{t_b}^{t_c} \mu(\tau)\, d\tau.
\end{eqnarray*}
\end{lemma}
\begin{proof}
Integrating identity (\ref{eq:errorsys_r}) with respect to time and using (\ref{id1}) we obtain
\begin{eqnarray*}
\fl r(t_c)-r(t_b)=\int^{t_c}_{t_b} r_t(\tau) d\tau \\
\fl =\int^{t_c}_{t_b}\left\{R(C(q^\ast,u^\ast)-C(q^\ast,u^\ast+p))-RA(u^\ast+p)e-\mu RM\frac{r}{\|r\|_{\Vtil}}\right\}d\tau \\
\end{eqnarray*}
Taking the norm we get, using the triangle inequality and the reverse triangle inequality,
\begin{eqnarray*}
\fl \|r(t_c)\|_X+\|r(t_b)\|_X \geq \|r(t_c)-r(t_b)\|_X \\
\fl \geq \|\int^{t_c}_{t_b}{RA(u^\ast+p)e}d\tau \|_X-\int^{t_c}_{t_b}{\|RC(q^\ast,u^\ast)-RC(q^\ast,u^\ast+p)\|_X}d\tau\\
\fl -\int^{t_c}_{t_b}{\|\mu RM\frac{r}{\|r\|_{\tilde{V}}}\|_X}d\tau \\
\fl \geq\|\int^{t_c}_{t_b}{RA(u^\ast+p)e}d\tau\|_X-\int^{t_c}_{t_b}{L_C \|p\|_{\hat{V}}}d\tau-\int^{t_c}_{t_b}{C_M \mu}d\tau
\end{eqnarray*}
where we have used Assumption \ref{ass:equation}.
Now we have to estimate the remaining first term on the right hand side.
With Assumption \ref{ass:equation} as well as $L_A$ as in (\ref{def:LA}) we get 
\begin{eqnarray*}
\fl \|\int^{t_c}_{t_b} RA(u^\ast(\tau)+p(\tau))e(\tau) d\tau\|_X \\
\fl =\|\int^{t_c}_{t_b} RA(u^\ast(\tau)+p(\tau))(e(t_a)+e(\tau)-e(t_a))d\tau \|_X \\
\fl \geq \|\int^{t_c}_{t_b} RA(u^\ast(\tau)+p(\tau))e(t_a)d\tau \|_X-\int^{t_c}_{t_b}{\|A(u^\ast+p)\|_{Q\rightarrow X}\|e(\tau)-e(t_a)\|_Q}d\tau \\
\fl \geq \|\int^{t_c}_{t_b}{RA(u^\ast(\tau)+p(\tau))e(t_a)}d\tau\|_X-\int^{t_c}_{t_b}{L_A^2\int^{\tau}_{t_a}{\|r(\sigma)\|_X}d\sigma}d\tau 
\end{eqnarray*}
where  we used the fact that with $\tau\geq t_a$ and (\ref{eq:errorsys_e})
$$
\fl \|e(\tau)-e(t_a)\|_Q =\| \int^{\tau}_{t_a} A(u^\ast+p)^{\ast}r d\sigma \|_Q 
\leq L_A\int^{\tau}_{t_a}\|r(\sigma)\|_X d\sigma . 
$$
Combining everything yields the assertion.
\end{proof}
Consider the right hand side in the estimate of Lemma \ref{lem:est_state}. While by Theorem \ref{th:state}, the negative terms containing $r$ will tend to zero as time tends to infinity, the first (positive) term enables us to enforce parameter convergence by means of a so-called persistence of excitation condition.
\begin{assumption}[Persistence of Excitation]\label{asspof}
There are $T_0,\varepsilon_0,\gamma_0,\underline{t}>0$ such that for all $t_a \geq \underline{t}$,  $\xi\in\partial B^Q_1(0)$ there exists a time instance $t_b\in[t_a,t_a+T_0]$ such that 
$$\|\int^{t_b+\gamma_0}_{t_b}RA(u^\ast(\tau)+p(\tau))\xi d\tau \|_X \geq \varepsilon_0\,.$$
\end{assumption}
To control the remaining terms 
$-L_C \int_{t_b}^{t_c}\|p(\tau)\|\, d\tau$ and
$-C_M \int_{t_b}^{t_c} \mu(\tau)\, d\tau$
on the right hand side of the estimate in Lemma \ref{lem:est_state}, we will combine the estimate 
\begin{equation}
\label{eq:est_er2}
\frac{1}{2} \left[\|e\|_Q^2+\|r\|_X^2\right]_{t_1}^{t_2}\leq-\frac{c_M}{2}\int^{t_2}_{t_1}{\mu \frac{\|r\|^2_{\widetilde{VX}}}{\|r\|_{\tilde{V}}}}d\tau =-\frac{c_M}{2}\int_{t_1}^{t_2}  \theta(\tau) \, d\tau,
\end{equation}
where
$$
\theta=\mu \frac{\|r\|_{\VXtil}^2}{\|r\|_{\Vtil} },
$$
resulting from (\ref{eq:est_er1}) 
with some link conditions
\begin{assumption}[Link conditions]\label{cond:link}
There exist $\lambda,\kappa\in[1,\infty)$, $T_\lambda,T_\kappa>0$ and $C_\lambda,C_\kappa>0$ such that for $\gamma_0$ as in Assumption \ref{asspof} the following holds.\\
For all $t\geq T_\lambda$  
$$C_\lambda\geq
\cases
{\left(\int_{t}^{t+\gamma_0}\left(\frac{\|p(\tau)\|_\Vhat^\lambda}{\theta(\tau)}\right)^{\frac{1}{\lambda-1}}\, d\tau\right)^{\frac{\lambda-1}{\lambda}}\quad \mbox{ if }\lambda>1\\
\sup_{\tau\in[t,t+\gamma_0]} \frac{\|p(\tau)\|_\Vhat}{\theta(\tau)}  \quad \mbox{ if }\lambda=1.}
$$
For all $t\geq T_\kappa$  
$$C_\kappa\geq
\cases{
\left(\int_{t}^{t+\gamma_0}\left(\frac{\mu(\tau)^\kappa}{\theta(\tau)}\right)^{\frac{1}{\kappa-1}}\, d\tau\right)^{\frac{\kappa-1}{\kappa}}\quad \mbox{ if }\kappa>1\\
\sup_{\tau\in[t,t+\gamma_0]} \frac{\mu(\tau)}{\theta(\tau)} \quad \quad \mbox{ if }\kappa=1.}
$$
\end{assumption}
\begin{remark}
Sufficient for Assumption \ref{cond:link} is the existence of some $\rho>0$ and a constant $C_{\rho}$ such that for all $t>0$
$$
\|p\|_{\hat{V}}\leq C_{\rho} \|r\|^\rho_X
$$
and existence of constants $c_{int}$ and $C_{int}$ respectively $c_\mu$ and $C_\mu$ such that for all $t>0$ the following interpolation estimate 
\begin{equation}\label{interpol}
c_{int}\|r\|_{\tilde{V}}\|r\|_X\leq \|r\|^2_{\widetilde{VX}} \leq C_{int}\|r\|_{\tilde{V}} \|r\|_X.
\end{equation}
and also the connecting estimate of $r$ and $\mu$ 
$$
c_\mu \mu \leq \|r\|_X^{\frac{1}{\kappa-1}}\leq C_\mu \mu
$$
holds. This can be seen as follows. 
\newline Since we want to estimate the integral $(\int^{t+\gamma_0}_t{(\frac{\mu(\tau)^\kappa}{\theta(\tau)})^{\frac{1}{\kappa-1}}}d\tau)^{\frac{\kappa -1}{\kappa}}$ we first take a look at the integrand.
Using the definition of $\theta$ and the stated interpolation estimate for the state error as well as the connecting estimate of $r$ and $\mu$ we get
$$
\left(\frac{\mu^\kappa}{\theta}\right)^{\frac{1}{\kappa-1}} \leq  \mu \left(\frac{1}{c_{int}\|r\|_X}\right)^{\frac{1}{\kappa-1}} \leq\frac{1}{c_{int}^{\frac{1}{\kappa-1}}c_{\mu}}
$$
and so the integral is
$$
\left(\int^{t+\gamma_0}_{t}{\left(\frac{\mu(\tau)^\kappa}{\theta(\tau)}\right)^{\frac{1}{\kappa-1}}}d\tau \right)^\frac{\kappa-1}{\kappa} \leq \left(\int^{t+\gamma_0}_{t}{\frac{1}{c_{int}^{\frac{1}{\kappa-1}}c_\mu}}d\tau \right)^{\frac{\kappa-1}{\kappa}}=\left(\frac{1}{c_{int}}\right)^{\frac{1}{\kappa}}\left(\frac{\gamma_0}{c_{\mu}}\right)^{\frac{\kappa-1}{\kappa}} \, \leq const.
$$
The second integral $\left(\int^{t+\gamma_0}_{t}{\left(\frac{\|p(\tau)\|^\lambda_{\hat{V}}}{\theta(\tau)}\right)^{\frac{1}{\lambda-1}}}d\tau \right)^{\frac{\lambda-1}{\lambda}}$ can be estimated similarly. Again using the definition of $\theta$ and the estimates stated in the remark yields 
$$
\frac{\|p\|^\lambda_{\hat{V}}}{\theta} \leq \frac{1}{c_{int}}\frac{\|p\|^\lambda_{\hat{V}}}{\mu \|r\|_X} \leq \frac{C_\mu C^\lambda_\rho}{c_{int}}\|r\|_X^{\lambda \rho -\frac{\kappa}{\kappa-1}}.
$$
Therewith the integral is bounded, using Proposition \ref{prop:welldefine}
\begin{eqnarray*}
\fl \left(\int^{t+\gamma_0}_{t}{\left(\frac{\|p(\tau)\|^\lambda_{\hat{V}}}{\theta(\tau)}\right)^{\frac{1}{\lambda-1}}}d\tau \right)^{\frac{\lambda-1}{\lambda}} \\
\leq \left(\frac{C_\mu}{c_{int}}\right)^\frac{1}{\lambda}C_\rho\left(\int^{t+\gamma_0}_{t}{\left(\|r(\tau)\|_X^{\lambda \rho -\frac{\kappa}{\kappa-1}}\right)^\frac{1}{\lambda-1}}d\tau \right)^\frac{\lambda-1}{\lambda} \leq const
\end{eqnarray*}
provided $\lambda \geq \frac{\kappa}{\rho(\kappa-1)}$.

A possible choice for $\kappa$ and $\lambda$ is to take $\lambda=\frac{\kappa}{(\kappa-1)\rho}$ and $\kappa=\max\{1+\frac{1}{\rho},2\} $, which arises from Assumption \ref{cond:mu02}.
\end{remark}
With these assumptions we can state the next lemma.
\begin{lemma}
\label{lemma1}
Let Assumptions \ref{ass:equation} with (\ref{AbdX}), \ref{cond:nu}, \ref{cond:mu02}, \ref{asspof}, and \ref{cond:link} hold.\\
Then, for any given $\gamma>0$, there are $\varepsilon>0$, $T>0$ and $T_1>0$ such that for all $t_1\geq T_1$ the following holds true:\\
If the parameter error $\|e(t_1)\|_Q \geq \gamma$, then there exists a $t_2\in [t_1,t_1+T]$ such that the state error $\|r(t_2)\|_X\geq \varepsilon$.
\end{lemma}
\begin{proof}
We choose $T_0$, $\varepsilon_0$, $\gamma_0$, $\underline{t}>0$ according to Assumption \ref{asspof}, fix $\gamma >0$ arbitrarily, set $T_1 =max \left\{\underline{t},\bar{t} \right\}$ and assume that $t_1 > T_1$ and $\|e(t_1)\|_Q > \gamma$.
(Here $\bar{t}$ will be chosen sufficiently large below.)  Setting $\xi=\frac{e(t_1)}{\|e(t_1)\|_Q}$ we can choose $t_b$ according to Assumption \ref{asspof}. Now we use Lemma \ref{lem:est_state} with $t_a=t_1$, $t_c=t_b+\gamma_0$, and set $t_2=t_c$ and $T=T_0+\gamma_0$ (i.e. $t_a=t_1 \leq t_b \leq t_b+\gamma_0 =t_2 =t_c \leq t_1+T=t_a+T$) to obtain
\begin{eqnarray*}
\fl \|r(t_2)\|_X =\|r(t_b+\gamma_0)\|_X\geq
\|\int^{t_b+\gamma_0}_{t_b} RA(u^\ast+p)\frac{e(t_1)}{\|e(t_1)\|_Q}d\tau\|_X \,\|e(t_1)\|_Q\\
\fl -\|r(t_b)\|_X
-L_A^2\int_{t_b}^{t_b+\gamma_0}\int_{t_1}^\tau \|r(\sigma)\|_X\, d\sigma d\tau 
-L_C \int_{t_b}^{t_b+\gamma_0}\|p(\tau)\|_\Vhat\, d\tau
-C_M \int_{t_b}^{t_b+\gamma_0} \mu(\tau)\, d\tau \\
\fl \geq \varepsilon_0 \|e(t_1)\|_Q-\|r(t_b)\|_X -L_A^2\int^{t_b+\gamma_0}_{t_b}{\int^{\tau}_{t_1}{\|r(\sigma)\|_X}d\sigma}d\tau \\
\fl -L_C\int^{t_b+\gamma_0}_{t_b}{\|p(\tau)\|_{\hat{V}}}d\tau-C_M\int^{t_b+\gamma_0}_{t_b}{\mu(\tau)}d\tau. 
\end{eqnarray*}
The last three terms remain to be estimated.
$$
\int^{t_b+\gamma_0}_{t_b}{\int^{\tau}_{t_1}{\|r(\sigma)\|_X}d\sigma}d\tau  \leq \int^{t_b+\gamma_0}_{t_b}{\int^{\tau}_{t_1}{\sup_{\sigma \geq t_1}\|r(\sigma)\|_X}d\sigma}d\tau \leq \gamma_0 T \sup_{\sigma \geq t_1}{\|r(\sigma)\|_X}.
$$
Estimating by H\"older's inequality and using the link conditions results in
$$
\int_{t_b}^{t_b+\gamma_0}\|p(\tau)\|\, d\tau =\int^{t_b+\gamma_0}_{t_b}{\frac{\|p(\tau)\|_{\hat{V}}}{\theta(\tau)^{\frac{1}{\lambda}}}\theta^{\frac{1}{\lambda}}(\tau)}d\tau \leq C_\lambda \left(\int^{t_b+\gamma_0}_{t_b}{\theta(\tau)}d\tau\right)^\frac{1}{\lambda}
$$
and analogously for the last term
$$
\int^{t_b+\gamma_0}_{t_b}{\mu(\tau)}d\tau =\int^{t_b+\gamma_0}_{t_b}{\frac{\mu(\tau)}{\theta(\tau)^{\frac{1}{\kappa}}}\theta^{\frac{1}{\kappa}}(\tau)}d\tau \leq C_\kappa \left(\int^{t_b+\gamma_0}_{t_b}\theta(\tau) d\tau\right)^\frac{1}{\kappa}.
$$
Now using (\ref{eq:est_er2}) we can estimate the term $\int_{t_b}^{t_b+\gamma_0} \theta(\tau)\, d\tau$  as follows
$$
\int_{t_b}^{t_b+\gamma_0}  \theta(\tau) \, d\tau  \leq -\frac{1}{c_M}[\|e\|^2_Q+\|r\|^2_X]^{t_b+\gamma_0}_{t_b}  \leq \frac{1}{c_M} (\|e(t_b)\|_Q^2-\|e(t_b+\gamma_0)\|_Q^2)+\frac{1}{c_M}\|r(t_b)\|_X^2 .
$$
At this point we utilize (\ref{eq:errorsys_e}), Assumptions \ref{ass:equation} and \ref{def:LA} as well as Proposition \ref{prop:welldefine}
\begin{eqnarray*}
\fl \|e(t_b)\|^2_Q-\|e(t_b+\gamma_0)\|^2_Q  = [\|e(t)\|^2_Q ]^{t_b}_{t_b+\gamma_0}=-2\int^{t_b+\gamma_0}_{t_b}{(e_t,e)_Q}d\tau  \\
\fl \leq 2\int^{t_b+\gamma_0}_{t_b}{\|e_t\|_Q\|e\|_Q}d\tau  =2\int^{t_b+\gamma_0}_{t_b}{\|A(u^\ast+p)^\ast r\|_Q \|e\|_Q}d\tau \\ \fl \leq 2L_A \int^{t_b+\gamma_0}_{t_b}{\|r(\tau)\|_X \|e(\tau)\|_Q}d\tau 
\leq 2L_A \int^{t_b+\gamma_0}_{t_b}{\|r(\tau)\|_X}d\tau \sqrt{\|e(0)\|^2_Q+\|r(0)\|^2_X}.
\end{eqnarray*}
Hence altogether we end up with 
\begin{eqnarray*}
\fl \|r(t_2)\|_X \geq
\epsilon_0\gamma-\|r(t_b)\|_X
-L_A^2\gamma_0 T \sup_{\sigma\geq t_1} \|r(\sigma)\|_X\\
\fl -L_C C_\lambda \left(\frac{2L_A}{c_M} \int_{t_b}^{t_b+\gamma_0} \|r(\tau)\|_X d\tau \sqrt{\|e(0)\|_Q^2+\|r(0)\|_X^2}
+\frac{1}{c_M}\|r(t_b)\|_X^2\right)^{\frac{1}{\lambda}} \\
\fl -C_M C_\kappa \left(\frac{2 L_A}{c_M} \int_{t_b}^{t_b+\gamma_0} \|r(\tau)\|_X d\tau \sqrt{\|e(0)\|_Q^2+\|r(0)\|_X^2}
+\frac{1}{c_M}\|r(t_b)\|_X^2\right)^{\frac{1}{\kappa}}.
\end{eqnarray*}
By Theorem \ref{th:state} for $\bar{t}$ sufficiently large, $t_b, \, t_1, \, \tau\geq T_1\geq\bar{t}$ the sum of all negative terms will be contained in the interval  $[-\frac{\epsilon_0\gamma}{2},0]$, so that we get 
$$
\|r(t_2)\|_X \geq \epsilon_0 \gamma - \frac{\epsilon_0\gamma}{2}=\frac{\epsilon_0\gamma}{2}\,.
$$ 
With $\epsilon=\frac{\epsilon_0\gamma}{2}$, this implies the assertion.
\end{proof}
\begin{theorem}[Parameter convergence]
\label{th:parameter}
Under Assumptions \ref{ass:equation} with (\ref{AbdX}), \ref{cond:nu}, \ref{cond:mu02}, \ref{asspof} and \ref{cond:link}
we have that 
$$
\left\|\hat{q}(t)-q^*\right\|_Q\rightarrow 0\mbox{ as }t \rightarrow \infty\,.
$$
\end{theorem}
\begin{proof}
Contraposition in Lemma \ref{lemma1} gives the following assertion (as we have imposed Assumptions \ref{ass:equation}, \ref{cond:nu}, \ref{cond:mu02}, \ref{asspof}, and \ref{cond:link} to hold):
\begin{quote}
For any given $\gamma >0$, there are $\varepsilon$, $T$, $T_1>0$ such that for all $t_1\geq T_1$ the following holds true:\\ 
If for all $t_2\in [t_1,t_1+T]$ the state error $\|r(t_2)\|_X< \varepsilon$, then the parameter error $\|e(t_1)\|_Q < \gamma$.
\end{quote}
Thus, given arbitrary $\gamma>0$, we choose $\varepsilon$ and $T_1$ according to Lemma \ref{lemma1}. Then, by Theorem \ref{th:state}, there exists $T_2\geq T_1$ such that for all $t\geq T_2$ we have $\|r(t)\|_X<\varepsilon$. Hence, for all $t_1\geq T_2$, the above statement yields $\|e(t_1)\|_Q < \gamma$.
\end{proof}
\subsection{Convergence with noisy data}\label{sec:noisy}
In case noisy data $\zdel$ are given instead of $z$ and the range of $G$ is non-closed, the quantity $G^\dagger \zdel$ might not be well-defined, and even if it is well-defined it will not depend on $\zdel$ in a stable manner. Thus we define a regularized version of the ``observed'' part of $u^*$
$$u_\alpha^\delta=G_\alpha \zdel$$
with $G_\alpha$ a regularized version of $G^\dagger$ with regularization parameter $\alpha$, defined, e.g., by the  Tikhonov-Philips method
$$G_\alpha=(G^*G+\alpha I)^{-1}G^*:Z\to \cN(G)^\bot\subseteq X$$
with $G^*:Z\to X$ the Hilbert space adjoint of $G:X\to Z$, and $\alpha>0$ appropriately chosen.
Additionally one might add a stabilizing term defined by another parameter $\sigma=\sigma(t)\geq 0$, see e.g. \cite{Ioannou}. Note that also the case $\sigma \equiv 0$ is included in our analysis. As a matter of fact, it turns out that this term is not really needed. For the sake of completeness to some extent we will also consider the case of strictly positive $\sigma$. The case of partially vanishing, partially positive $\sigma$ is not included here (but could be approximated by some positive $\sigma$ which partially gets arbitrarily small).

Therewith, we redefine the estimators $\hat{q}$, $\hat{u}$ by 
\begin{eqnarray}
\hat{q}_t-A(u_\alpha^\delta+P\hat{u})^*(R\hat{u}-u_\alpha^\delta)=-\sigma \hat{q} \label{eq:noisy_q} \\
\hat{u}_t+C(\hat{q},u_\alpha^\delta+P\hat{u})+\mu R M \frac{R\hat{u}-u_\alpha^\delta}{\|R\hat{u}-u_\alpha^\delta\|}_{\Vtil}+\nu PN P\hat{u}=f \label{eq:noisy_u}\\
(\hat{q},\hat{u})(0)=(\hat{q}_0,\hat{u}_0) \label{eq:noisy_0}
\end{eqnarray}
where $\alpha=\alpha(t)$, $\mu=\mu(t)$ and $\nu=\nu(t)$ are chosen properly dependent on the noise level $\delta(t)$ in
\begin{equation} \label{delta} 
\delta(t) \geq \| z^\delta(t)-z(t)\|_Z.
\end{equation}
\subsubsection{Well-definedness}
For showing well-definedness we take again a look at the error components
$e=\hat{q}-q^\ast$, $r=R\hat{u}-Ru^\ast$,  $p=P\hat{u}-Pu^\ast$ and the errors including the regularized version of the ``observed'' part 
\begin{equation}\label{id2}
r^\delta_\alpha=R\hat{u}-u^\delta_\alpha=r-d^\delta_\alpha \mbox{ and } d^\delta_\alpha=u^\delta_\alpha-Ru^\ast.
\end{equation}
Therewith the equalities 
\begin{equation}\label{id3}
 u^\delta_\alpha+P\hat{u}=u^\delta_\alpha+P\hat{u}+u^\ast-Ru^\ast-Pu^\ast=u^\ast+d^\delta_\alpha+p
\end{equation} hold.
Then the differential equations for the errors are
\begin{eqnarray}
\fl e_t-A(u^\ast+d^\delta_\alpha+p)^\ast r^\delta_\alpha= -\sigma \hat{q}\label{eq:smooth_err_q} \\
\fl r_t+RC(q^\ast,u^\ast+d^\delta_\alpha+p)-RC(q^\ast,u^\ast)+RA(u^\ast+d^\delta_\alpha+p)e+\mu RM \frac{r^\delta_\alpha}{\|r^\delta_\alpha\|_{\tilde{V}}}=0 \label{eq:smooth_err_r}\\
\fl p_t+PC(q^\ast,u^\ast+d^\delta_\alpha+p)-PC(q^\ast,u^\ast)+PA(u^\ast+d^\delta_\alpha+p)e+\nu PNP\hat{u}=0 \label{eq:smooth_err_p}\\
\fl (e,r,p)(0)=(\hat{q}_0-q^\ast,R(\hat{u}_0-u_0), P(\hat{u}_0-u_0)). \label{eq:smooth_err_0}
\end{eqnarray}
In case of noisy data we get a wellposedness result too. As in the exact data case some assumptions concerning the parameters $\mu$ and $\nu$ are required.
\begin{assumption}\label{ass:mu_nu}For all $t>0$
\begin{enumerate}
\item \begin{eqnarray*}
\fl \mu(t)\geq \max\Bigg\{\frac{4 L_C}{c_M }(\|d^\delta_\alpha(t)\|_\Vtil+\|p(t)\|_\Vhat)\|r(t)\|_X  \\
\fl +\frac{4 C_A}{c_M}(1+\|d^\delta_\alpha(t)\|_\Vtil + \|p(t)\|_\Vhat)\|e(t)\|_Q \|d^\delta_\alpha(t)\|_X , \frac{2 \sigma(t)}{c_M} \|r(t)\|^2_X \Bigg\}\frac{\|r^\delta_\alpha (t)\|_\Vtil}{\|r(t)\|^2_\VXtil}. 
\end{eqnarray*}
\item \begin{eqnarray*}
\fl \nu(t)\geq \max \Big\{\underline{\nu}, \Big( \frac{4(L_C+C_A\|e(t)\|_Q)}{c_N}(\|p(t)\|_\Vhat+\|d^\delta_\alpha\|_\Vtil)\\
+\frac{2 C_A C_{\Vhat \VXhat}C_{\VXhat X}}{c_N}\|p(t)\|_\Vhat\Big)\frac{\|p(t)\|_X}{\|p(t)|^2_\VXhat}\Big\}.
\end{eqnarray*}
\end{enumerate}
\end{assumption}

A condition on the error between the regularized version of the ``observed'' part and the exact state is also needed, namely for all considered time instances $t$
\begin{equation} \label{cond_d}
\|d^\delta_\alpha(t)\|_\Vtil \leq \frac{c_M}{2 C_M}\frac{\|r(t)\|^2_\VXtil}{\|r(t)\|_X}.
\end{equation}
should hold.
This condition on smallness can be further accessed using the fact that $d^\delta_\alpha=G_\alpha z^\delta-G^\dagger z$ and (\ref{delta}), based on results of regularization theory and an appropriate choice of $\alpha(t)$ in dependence of $\delta(t)$ and $z^\delta(t)$, see, e.g. \cite{EHNbuch}. We now prove that $\hat{q}$ and $\hat{u}$ according to (\ref{eq:noisy_q}) and (\ref{eq:noisy_u}) are well defined at least up to a certain time.
\begin{proposition} \label{well_defined_noisy}
Let Assumptions \ref{ass:equation} with (\ref{AbdVV}) and \ref{ass:mu_nu} hold and let $(\hat{q}_0-q^\ast,\hat{u}_0-u_0)\in Q \times (\tilde{V}+\hat{V})$. Then there exists a solution $(\hat{q}(t),\hat{u}(t))\in Q \times (\tilde{V}+\hat{V})$ of (\ref{eq:noisy_q})-(\ref{eq:noisy_0}) for all times $0<t<T^\ast$ where 
\begin{equation}
\label{Tast}
T^\ast=\min \left\{t>0 : \|d^\delta_\alpha\|_\Vtil >\frac{c_M}{2 C_M}\frac{\|r\|^2_\VXtil}{\|r\|_X} \right\}.
\end{equation}
(i.e. the first time, when condition (\ref{cond_d}) is violated) and satisfies the following error bounds (cf. (\ref{erp})).
\begin{enumerate}
\item Case: $\sigma \equiv0$: For all $ 0<t<T^\ast$: $\|e(t)\|^2_Q+\|r(t)\|^2_X \leq \|e(0)\|^2_Q+\|r(0)\|^2_X< \infty$;

Case: $\sigma>0$: For all $0<t<T^\ast$: 

$\|e(t)\|^2_Q+\|r\|_X^2\leq \max\left\{\|q^\ast\|^2_Q,\|e(0)\|^2_Q+\|r(0)\|^2_X\right\}< \infty$;
\item For all $ \, 0<t<T^\ast$: 

$\|p(t)\|^2_X \leq 2 \left\{\frac{C_N^2 C^4_{\VXtil X}}{c_N^2}\sup_{t>0}{\|Pu^\ast(t)\|^2_\Vtil}+\frac{C_A C^2_{\VXtil X}}{c_N \underline{\nu}}[\|e(0)\|^2_Q+\|r(0)\|^2_X] , \|p(0)\|^2_X \right\}$;
\item If $T^\ast=\infty$  (cf. (\ref{erp})) and $\sigma \equiv 0$ then $\int^{\infty}_{0}{\|p(\tau)\|_\Vtil\|r(\tau)\|_X}d\tau\leq \frac{\|e(0)\|^2_Q \|r(0)\|^2_X}{2L_C}$.
\end{enumerate}
\end{proposition}
\begin{proof}
\textit{1.} For proving the proposition, like in the exact data case we take a look at the norms of the squared errors.
\begin{eqnarray*}
\fl \frac{d}{dt}\frac{1}{2}[\|e\|^2_Q+\|r\|^2_X]=(e_t,e)_Q+(r_t,r)_X \\
\fl=\underbrace{\left(A(u^\ast+d^\delta_\alpha+p)^\ast r^\delta_\alpha,e \right)_Q- \left(RA(u^\ast+d^\delta_\alpha+p)e,r \right)_X}_{(1)} \\
\fl-\underbrace{\left(\sigma \hat{q},e\right)_Q}_{(2)} +\underbrace{\left(RC(q^\ast,u^\ast)-RC(q^\ast,u^\ast+d^\delta_\alpha+p),r \right)_X}_{(3)}-\underbrace{\mu \left(RM\frac{r^\delta_\alpha}{\|r^\delta_\alpha\|_{\tilde{V}}},r \right)_X}_{(4)}
\end{eqnarray*}
Our goal is now to estimate all these terms appropriately.\\
ad $(1)$:
Using the identity $r=r^\delta_\alpha+d^\delta_\alpha$ and Assumption \ref{ass:equation} we get
\begin{eqnarray*}
\fl \left(A(u^\ast+d^\delta_\alpha+p)^\ast r^\delta_\alpha,e \right)_Q-\left(RA(u^\ast+d^\delta_\alpha+p)e,r^\delta_\alpha+d^\delta_\alpha \right)_X =-\left(RA(u^\ast+d^\delta_\alpha+p)e,d^\delta_\alpha \right)_X \\
\fl=-\left(RA(u^\ast+d^\delta_\alpha+p)e,d^\delta_\alpha \right)_X  \leq C_A (1+\|d^\delta_\alpha\|_{\tilde{V}}+\|p\|_ {\hat{V}} )\|e\|_Q \|d^\delta_\alpha\|_X
\end{eqnarray*}
ad $(2)$: With some computations we get
\begin{eqnarray*}
\fl-(\sigma \hat{q},e)_Q = -\sigma (\hat{q}\pm q^\ast,e)_Q = \sigma(q^\ast,e)_Q-\sigma(e,e)_Q \\
\fl \leq -\sigma \|e\|^2_Q+\sigma\|q^\ast\|_Q \|e\|_Q \leq -\sigma \|e\|^2_Q +\frac{\sigma}{2}(\|q^\ast\|^2_Q+\|e\|^2_Q)=\frac{\sigma}{2}\|q^\ast\|^2_Q-\frac{\sigma}{2}\|e\|^2_Q.
\end{eqnarray*}
ad $(3)$: The Lipschitz condition on $C$ yields 
\begin{eqnarray*}
\fl \left(RC(q^\ast,u^\ast)-RC(q^\ast,u^\ast+d^\delta_\alpha+p),r \right)_X  \leq \|C(q^\ast,u^\ast)-C(q^\ast,u^\ast+d^\delta_\alpha+p)\|_X \|r\|_X \\
\fl \leq L_C (\|d^\delta_\alpha\|_{\tilde{V}}+\|p\|_{\hat{V}})\|r\|_X.
\end{eqnarray*}
ad $(4)$: Using coercivity and boundedness of $M$ (Assumption \ref{ass:equation}) and $t\leq T^\ast$ with $T^\ast$ as in (\ref{Tast}) results in
\begin{eqnarray}
\fl -\mu \left(RM \frac{r^\delta_\alpha}{\|r^\delta_\alpha\|_{\tilde{V}}},r \right)_X=-\frac{\mu}{\|r^\delta_\alpha\|_{\tilde{V}}}\left(RMr,r \right)_X+\frac{\mu}{\|r^\delta_\alpha\|_{\tilde{V}}}\left(RMd^\delta_\alpha,r\right)_X \nonumber \\
\fl \leq -\mu c_M \frac{\|r\|^2_{\widetilde{VX}}}{\|r^\delta_\alpha\|_{\tilde{V}}}+\mu C_M \frac{\|d^\delta_\alpha\|_{\tilde{V}}\|r\|_X}{\|r^\delta_\alpha\|_{\tilde{V}}}\leq
-\frac{\mu}{2}c_M \frac{\|r\|^2_{\widetilde{VX}}}{\|r^\delta_\alpha\|_{\tilde{V}}} \label{est5}
\end{eqnarray}
Inserting in the original inequality gives
\begin{eqnarray}
\fl \frac{d}{dt}\frac{1}{2}[\|e\|^2_Q+\|r\|^2_X] \leq C_A(1+\|d^\delta_\alpha\|_\Vtil+\|p\|_\Vhat)\|e\|_Q\|d^\delta_\alpha\|_X +\frac{\sigma}{2}\|q^\ast\|^2_Q-\frac{\sigma}{2}\|e\|^2_Q \nonumber \\
\fl +L_C(\|d^\delta_\alpha\|_\Vtil+\|p\|_\Vhat)\|r\|_X-\frac{\mu}{2}c_M \frac{\|r\|^2_\VXtil}{\|r^\delta_\alpha\|_\Vtil} \label{est4}
\end{eqnarray}
Using Assumption \ref{ass:mu_nu} on $\mu$ we get
\begin{eqnarray*}
\fl \frac{d}{dt}\frac{1}{2}[\|e\|^2_Q+\|r\|_X^2]\leq \frac{\sigma}{2}\|q^\ast\|^2_Q-\frac{\sigma}{2}\|e\|^2_Q-\frac{\mu c_M}{4}\frac{\|r\|^2_\VXtil}{\|r^\alpha_\delta\|_\Vtil} \leq \frac{\sigma}{2}(\|q^\ast\|^2_Q-(\|e\|^2_Q+\|r\|^2_X)). 
\end{eqnarray*}
Now we distinguish between the two cases $\sigma\equiv 0$ and $\sigma >0$.
For the first case $\sigma =0$ we have
$$
\frac{d}{dt}\frac{1}{2}[\|e\|^2_Q+\|r\|^2_X]\leq 0 \, \Rightarrow \, \frac{1}{2}[\|e\|^2_Q+\|r\|^2_X]\leq \frac{1}{2}[\|e(0)\|^2_Q+\|r(0)\|^2_X].
$$
For the second case $\sigma >0$ we define $\tau (t):= \int^t_0 {\sigma(\xi)}d \xi$, $\mathcal{V}(t):=\frac{1}{2}[\|e(t)\|^2_Q+\|r(t)\|^2_X]$ and $\tilde{\mathcal{V}}(\tau(t)):=\mathcal{V}(t)$.
Differentiating $\tilde{\mathcal{V}}$ with respect to $\tau$ leads to
$$
\frac{d}{d\tau}\tilde{\mathcal{V}}(\tau(t))=\frac{1}{2}[\|e\|^2_Q+\|r\|^2_X]\frac{1}{\sigma(t)}
 \leq \frac{1}{2}\|q^\ast\|^2_Q-\frac{1}{2}[\|e\|^2_Q+\|r\|^2_X]= \frac{1}{2}\|q^\ast\|^2_Q-\tilde{\mathcal{V}}(\tau(t)).
$$
So we have for all $t>0$
$$
\frac{1}{2}[\|e\|^2_Q+\|r\|^2_X]\leq \max \left\{\frac{1}{2}\|q^\ast\|^2_Q, \frac{1}{2}[\|e(0)\|^2_Q+\|r(0)\|^2_X]\right\} < \infty.
$$
\textit{2.} We now consider the error for the ``unobserved'' part of the state. Similarly to (\ref{estp}) we get
\begin{eqnarray*}
\fl \frac{d}{dt}\frac{1}{2}[\|p\|^2_X]= (p_t,p)_X\\
\fl-\left(PC(q^\ast,u^\ast+d^\delta_\alpha+p)-PC(q^\ast,u^\ast),p\right)_X-\left(PA(u^\ast+d^\delta_\alpha+p)e,p \right)_X-\nu \left( PNP\hat{u},p \right)_X \\
\fl \leq L_C (\|d^\delta_\alpha\|_\Vtil+\|p\|_\Vhat)\|p\|_X+C_A(1+\|d^\delta_\alpha\|_\Vtil+\|p\|_\Vhat)\|e\|_Q\|p\|_X \\
\fl -\frac{\nu c_N}{2} \|p\|^2_{\widehat{VX}}+\nu \frac{C_N^2 C^2_{\widehat{VX}X}}{2c_N}\|Pu^\ast\|^2_{\hat{V}} \\
\fl\leq \frac{C_A}{2}(\|e\|^2_Q+C_{\VXhat X}C_{\Vhat \VXhat}\|p\|_{\Vhat}\|p\|_X)+(L_C+C_A\|e\|_Q)(\|d^\delta_\alpha\|_\Vtil+\|p\|_\Vhat)\|p\|_X \\
\fl-\frac{\nu c_N}{2}\|p\|^2_\VXhat+\frac{\nu C^2_N C^2_{\VXhat X}}{2c_N}\|Pu^\ast\|_\Vhat^2 \\
\fl= \left[(L_C+C_A\|e\|_Q)(\|d^\delta_\alpha\|_\Vtil+\|p\|_\Vhat)+\frac{C_A C_{\VXhat X}C_{\Vhat \VXhat}}{2}\|p\|_\Vhat\right]\|p\|_X \\
 \fl-\frac{\nu c_N}{2}\|p\|^2_\VXhat+\frac{C_A}{2}\|e\|^2_Q+\frac{\nu C^2_N C^2_{\VXhat X}}{2 c_N}\|Pu^\ast\|^2_\Vhat.
\end{eqnarray*}
Here we have used Assumption \ref{ass:equation} and (\ref{est1}).
Now we make use of Assumption \ref{ass:mu_nu} on $\nu$ to get
$$
\frac{d}{dt}\frac{1}{2}[\|p\|^2_X]\leq -\frac{\nu c_N}{4}\|p\|^2_{\VXhat}+\nu \frac{C_N^2C^2_{\VXhat X}}{2c_N}\|Pu^\ast\|^2_\Vhat+\frac{C_A}{2}\|e\|^2_Q.
$$
We again define functions $\tau(t):=\frac{c_N}{2 C^2_{\VXhat X}}\int^t_0{\nu(\xi)}d\xi$,  $\mathcal{V}(t):=\frac{1}{2}\|p(t)\|^2_X$ and $\tilde{\mathcal{V}}(\tau(t)):= \mathcal{V}(t)$.
Differentiating $\tilde{\mathcal{V}}$ with respect to $\tau$ leads to
\begin{eqnarray*}
\frac{d}{d\tau}\tilde{\mathcal{V}}(\tau)
\leq (-\frac{\nu c_N}{4}\|p\|^2_\VXhat+\nu\frac{C^2_N C^2_{\VXhat X}}{2 c_N}\|Pu^\ast\|^2_\Vhat+\frac{C_A}{2}\|e\|^2_Q)\frac{2 C^2_{\VXhat X}}{c_N \nu(t)}\\
=-\frac{C^2_{\VXhat X}}{2}\|p\|^2_\VXhat+\frac{C_N^2 C^4_{\VXhat X}}{c_N^2}\|Pu^\ast\|^2_\Vhat+\frac{C_A C^2_{\VXhat X}}{c_N \nu(t)}\|e\|^2_Q. 
\end{eqnarray*}
Using the embedding $\VXhat \hookrightarrow X$ and the estimate for $\nu$ in Assumption \ref{ass:mu_nu} gives
$$
\frac{d}{d\tau}\tilde{\mathcal{V}}(\tau) \leq
-\tilde{\mathcal{V}}(\tau(t))+\frac{C^2_N C^4_{\VXhat X}}{c_N^2}\|Pu^\ast\|^2_\Vhat +\frac{C_A C^2_{\VXhat X}}{c_N \underline{\nu}}\|e\|^2_Q.
$$
From this we get the assertion.\\
\textit{3.} This is a consequence of inequality (\ref{est4}) and Assumption \ref{ass:mu_nu}.
\begin{eqnarray*}
\fl \int^{\infty}_0{\|p(\tau)\|_\Vtil \|r(\tau)\|_X}d\tau \leq \frac{c_M}{4}\int^\infty_0{\mu(\tau)\frac{\|r(\tau)\|^2_\VXtil}{\|r^\delta_\alpha(\tau)\|_\Vtil}}d\tau \leq \int^{\infty}_0{\frac{d}{dt}\frac{1}{2}[\|e(\tau)\|^2_Q+\|r(\tau)\|^2_X]}d\tau \\
\fl \leq \frac{1}{2}[\|e(0)\|^2_Q+\|r(0)\|^2_X].
\end{eqnarray*}
\end{proof}
\subsubsection{State convergence}
As in the exact data case we introduce an additional lower bound on $\mu$ for proving  convergence of the ``observed'' part of the state estimate.
\begin{assumption}
\label{ass:mu_noisy}
There exists a constant $\tilde{c}_1>0$ such that for all $t>0$ 
\begin{eqnarray*}
\fl \mu(t)\geq \max\Big\{\frac{4L_C}{c_M}(\|d^\delta_\alpha(t)\|_\Vtil+\|p(t)\|_\Vhat)\|r(t)\|_X \\
 +\frac{4C_A}{c_M}(1+\|d^\delta_\alpha(t)\|_\Vtil+\|p(t)\|_\Vhat)\|e(t)\|_Q\|d^\delta_\alpha\|_X, \tilde{c}_1 \|r(t)\|^2_X \Big\} \frac{\|r^\delta_\alpha(t)\|_\Vtil}{\|r(t)\|^2_\VXtil}.
\end{eqnarray*}
\end{assumption}
\begin{theorem}[State convergence]
Under Assumptions \ref{ass:equation} with (\ref{AbdVV}) and \ref{ass:mu_noisy} and if $T^\ast=\infty$ (cf. (\ref{erp})) and $\sigma \equiv 0$ we have that $\|R(\hat{u}(t)-u^\ast(t))\|_X=\|r(t)\|_X \rightarrow 0$ as $t\rightarrow \infty$.
\end{theorem}
\begin{proof}
The proof is quite similar to the one in the exact data case. We start with considering the ``observed'' state error for $t_2 > t_1 >0$.

\begin{eqnarray*}
\fl \|r(t_2)\|^2_X-\|r(t_1)\|^2_X=\int^{t_2}_{t_1}(r_t,r)_X\\
\fl \int^{t_2}_{t_1}\Big(RC(q^\ast,u^\ast)-RC(q^\ast,u^\ast+d^\delta_\alpha +p)  -RA(u^\ast+d^\delta_\alpha+p)e-\frac{\mu}{\|r^\delta_\alpha\|_\Vtil}RM r^\delta_\alpha,r\Big)_X d\tau \\
\fl \leq \int^{t_2}_{t_1}\left\{ L_C (\|d^\delta_\alpha\|_{\tilde{V}}+\|p\|_\Vhat)\|r\|_X+C_A (1+\|p\|_X+\|d^\delta_\alpha\|_X)\|e\|_Q\|r\|_X -\frac{\mu c_M}{2}\frac{\|r\|^2_\VXtil}{\|r^\delta_\alpha\|_\Vtil}\right\} d\tau
\end{eqnarray*}
Here we have used the identities (\ref{id2}) and Assumptions \ref{ass:equation} and $T^\ast=\infty$.
Furthermore we will denote 
$$\tilde{L}_A:=C_A(1+\sup_{t>0}\left\{\|d^\delta_\alpha\|_X+\|p\|_X\right\}).$$
Assumption \ref{ass:mu_noisy} and Propostition \ref{well_defined_noisy} give us
\begin{eqnarray*}
\fl \|r(t_2)\|^2_X-\|r(t_1)\|_X^2 \\
\fl \leq \int^{t_2}_{t_1}\{    L_C (\|d^\delta_\alpha\|_\Vtil+\|p\|_\Vhat)\|r\|_X+\tilde{L}_A\|e\|_Q \|r\|_X-2L_C (\|d^\delta_\alpha\|_\Vtil+\|p\|_\Vhat)\|r\|_X\}d\tau  \\
\fl \leq \int^{t_2}_{t_1}{\frac{\tilde{L}_A}{2}(\|e\|^2_Q+\|r\|^2_X)}d\tau  \\
\fl \leq \int^{t_2}_{t_1}{\frac{\tilde{L}_A}{2}(\|e(0)\|^2_Q+\|r(0)\|^2_X)}d\tau =\tilde{c}_2(t_2-t_1) \,,
\end{eqnarray*}
where we have defined $\tilde{c}_2:=\frac{\tilde{L}_A}{2}\left\{\|e(0)\|^2_Q+\|r(0)\|^2_X\right\}$.

As in the exact data case (cf (\ref{ineq:gamma})) we get for any fixed $t$, $\gamma>0$ 
$$
\int^t_{t-\gamma}\|r(\tau)\|^2_X d\tau \geq \gamma \|r(t)\|^2_X-\frac{\tilde{c_2}}{2}\gamma^2.
$$
For $\sigma\equiv0$ the proof from now on is exactly the same as in the exact data case.
\end{proof}
\begin{remark}
If in (\ref{erp}) $T^\ast <\infty$ we cannot expect convergence of the state error to zero if $\delta >0$. However in this case the definition of $T^\ast$ implies 
$$
\|d^\delta_\alpha(T^\ast)\|_\Vtil > \frac{c_M}{2 C_M}\frac{\|r(T^\ast)\|^2_\VXtil}{\|r(T^\ast)\|_X}
$$
and therefore that $r(T^\ast)$ is small, namely in case the interpolation inequality (\ref{interpol}) holds we even have that at time $T^\ast$ the ``observed'' state error is (up to a constant factor $\frac{2 C_M}{c_{int}c_M}$) as small as the error in the ``observed'' state, both of them in the $\Vtil$-norm. 
\end{remark}
\subsubsection{Parameter convergence}
For proving that the estimated parameter converges to the exact one we again need two Lemmas.
\begin{lemma}\label{lemma_integral} Under Assumption \ref{ass:equation} with (\ref{AbdX}) the projected state errors $r$ and $p$ satisfy the following relation for all $0< t_a \leq t_b \leq t_c$:
\begin{eqnarray*}
\fl \|r(t_c)\|_X \geq \|\int^{t_c}_{t_b}RA(u^\ast+d^\delta_\alpha +p)e(t_a)d\tau\|_X - \|r(t_b)\|_X-\tilde{L}^2_A \int^{t_c}_{t_b}\int^{\tau}_{t_a}\|r^\delta_\alpha\|_X ds d\tau \\
\fl -\tilde{L}_A\int^{t_c}_{t_b}\int^{\tau}_{t_a}\sigma \|\hat{q}\|_Q ds d\tau -L_C \int^{t_c}_{t_b}(\|d^\delta_\alpha\|_\Vtil+\|p\|_\Vhat)d\tau-C_M\int^{t_c}_{t_b} \mu d\tau.
\end{eqnarray*}
\end{lemma}
\begin{proof}
The proof is basically the same as in the exact data case with $\|d^\delta_\alpha\|_\Vtil+\|p\|_\Vhat$ instead of $\|p\|_\Vhat$ and $\|r^\delta_\alpha\|_X$ instead of $\|r\|_X$ in the term with $\tilde{L}_A$ and the additional term with $\sigma$.
\end{proof}
The persistence of excitation condition is nearly the same as in the exact data case, except that we have $u^\delta_\alpha$ instead of $Ru^\ast$, i.e., here we have the regularized version of the ``observed'' part of the state.
\begin{assumption}[Persistence of Excitation]\label{ass:persistence_noisy}
There are $T_0, \varepsilon_0, \gamma_0, \underline{t}>0$ such that for all $t_a \geq \underline{t}$, $\xi \in \partial B^Q_1(0)$ there exists a time instance $t_b\in [t_a,t_a+T_0]$ such that
$$
\|\int^{t_b+\gamma_0}_{t_b}RA(u^\ast+d^\delta_\alpha(\tau)+p(\tau))\xi d\tau\|_X\geq \varepsilon_0.
$$
\end{assumption}
Also the link conditions are quite similar. 
With a slightly different definition of theta
\begin{equation}
\label{thetatil}
\tilde{\theta}:=\mu \frac{\|r\|^2_\VXtil}{\|r_\alpha^\delta\|_\Vtil}
\end{equation}
and involving the error $d^\delta_\alpha$ between the exact ``observed'' part and its regularized version we use the following link conditions.
\begin{assumption}[Link conditions]\label{ass:link_noisy}
There exist $\tilde{\lambda}$, $\tilde{\kappa}\in [1,\infty)$, $T_{\tilde{\lambda}}$, $T_{\tilde{\kappa}}>0$ and $C_{\tilde{\lambda}}$, $C_{\tilde{\kappa}}>0$ such that for $\gamma_0 >0$ as in Assumption \ref{ass:persistence_noisy} the following holds. 

For all $t \geq T_{\tilde{\lambda}}$
$$
C_{\tilde{\lambda}}\geq \cases{ \left( \int^{t+\gamma_0}_t{\left(\frac{(\|d^\delta_\alpha(\tau)\|_\Vtil+\|p(\tau)\|_\Vhat)^{\tilde{\lambda}}}{\tilde{\theta}(\tau)}\right)^{\frac{1}{\tilde{\lambda}-1}}}d\tau \right)^\frac{\tilde{\lambda}-1}{\tilde{\lambda}} \, \quad \mbox{  if  } \tilde{\lambda}>1 \\ \sup_{\tau \in [t,t+\gamma_0]}{\frac{\|d^\delta_\alpha(\tau)\|_\Vtil+\|p(\tau)\|_\Vhat}{\tilde{\theta}(\tau)}} \, \quad \quad \mbox{  if  } \tilde{\lambda}=1.}
$$
For all $t \geq T_{\tilde{\kappa}}$  
$$
C_{\tilde{\kappa}}\geq \cases{ \left(\int^{t+\gamma_0}_t {\left(\frac{\mu^{\tilde{\kappa}}(\tau)}{\tilde{\theta}(\tau)}\right)^{\frac{1}{\tilde{\kappa}-1}}}d\tau\right)^{\frac{\tilde{\kappa}-1}{\tilde{\kappa}}} \, \quad \mbox{  if  } \tilde{\kappa}>1 \\ \sup_{\tau \in [t,t+\gamma_0]}{\frac{\mu(\tau)}{\tilde{\theta}(\tau)}} \, \quad \quad \mbox{  if  } \tilde{\kappa}=1. }
$$
\end{assumption}
Furtheron we just consider the case $\sigma=0$. In the other case $\sigma>0$ we cannot prove  parameter convergence.
The second lemma that is needed for parameter convergence is exactly the same as in the exact data case. (cf Lemma \ref{lemma1})
\begin{lemma}
\label{lemma1_noisy}
Let Assumptions \ref{ass:equation} with (\ref{AbdX}), \ref{ass:mu_nu}, \ref{ass:mu_noisy}, \ref{ass:persistence_noisy}, and \ref{ass:link_noisy} hold and $\sigma \equiv 0$.
Then, for any given $\gamma>0$, there are $\varepsilon>0$, $T>0$ and $T_1>0$ such that for all $t_1\geq T_1$ the following holds true:\\
If the parameter error $\|e(t_1)\|_Q \geq \gamma$, then there exists a $t_2\in [t_1,t_1+T]$ such that the state error $\|r(t_2)\|_X\geq \varepsilon$.
\end{lemma}
\begin{proof}
In case $\sigma \equiv 0$ Lemma \ref{lemma_integral} with Assumption \ref{ass:persistence_noisy} gives the same estimate as in the exact data case with the only difference that we have to replace $L_A$ with $\tilde{L}_A$ and $\|p(\tau)\|_\Vhat$ by $(\|p(\tau)\|_\Vhat+\|d^\delta_\alpha\|_\Vtil)$. Thus with the adaptations we have made in the definition of $\tilde{\theta}$ and in the link conditions \ref{ass:link_noisy}, the proof obviously goes through like the one of Lemma \ref{lemma1}. 
\end{proof}

For $T^\ast=\infty$ we can prove parameter convergence analogously to Theorem \ref{th:parameter}.
\begin{theorem}[Parameter convergence]
\label{th:parameter_noisy}
Under Assumptions \ref{ass:equation} with (\ref{AbdX}), \ref{ass:mu_nu}, \ref{ass:mu_noisy}, \ref{ass:persistence_noisy}, and \ref{ass:link_noisy} and if $T^\ast=\infty$ (cf. (\ref{erp})), $\sigma \equiv 0$
we have that 
$$
\left\|\hat{q}(t)-q^*\right\|_Q\rightarrow 0\mbox{ as }t \rightarrow \infty\,.
$$
\end{theorem}
\begin{proof} See exact data case.
\end{proof}
\begin{remark}
In case $T^\ast<\infty$ we cannot prove parameter convergence, because in the persistence of excitation assumption we need to have $t\rightarrow  \infty$. 
\end{remark}
Since in case $T^\ast  < \infty$ we cannot completely prove convergence for the noisy data case we also take a look at the smoothed noisy data case.
\subsection{Convergence with smoothed noisy data}
With smoothed noisy data we denote $\zdel$ that is smooth with respect to time (which can be achieved by averaging over sufficiently large time intervals), i.e.
$$
\|(\zdel-z)_t(t)\|_Z=\|\zdel_t(t)-z_t(t)\|_Z\leq\tilde{\delta}(t)\,.
$$
That means for the error of the ``observed'' part of the state
$$
r_t=[r^\delta_\alpha+d^\delta_\alpha]_t=r^\delta_{\alpha t}+[G_\alpha z^\delta-G^\dagger z]_t = r^\delta_{\alpha t}+\tilde{d}^\delta_\alpha 
$$
with $\tilde{d}_\alpha^\delta:=G_\alpha(z^\delta_t-z_t)+(G_\alpha-G^\dagger)z_t+\alpha_t(\frac{d}{d \alpha} G_\alpha)z^\delta$.
Therewith the online parameter identification method as in Section \ref{sec:noisy} is given by
\begin{eqnarray}
\fl \hat{q_t}-A(u^\ast+d^\delta_\alpha+p)^\ast(R\hat{u}-u^\delta_\alpha)=0 \label{method_smooth_q}\\
\fl \hat{u}_t+C(\hat{q},u^\ast+d^\delta_\alpha+p)+\mu RM\frac{R\hat{u}-u^\delta_\alpha}{\|R\hat{u}-u^\delta_\alpha\|_\Vtil}+\nu PNP\hat{u}=f \label{method_smooth_u}\\
\fl (\hat{q},\hat{u})(0)=(\hat{q}_0,\hat{u}_0) \label{method_smooth_0}
\end{eqnarray}
Hence we can alternatively to (\ref{eq:smooth_err_q})-(\ref{eq:smooth_err_0}) consider the equations 
\begin{eqnarray}
\fl e_t-A(u^\ast+d_\alpha^\delta+p)^\ast r^\delta_\alpha=0 
 \nonumber \\
\fl r^\delta_{\alpha t}+\tilde{d}^\delta_\alpha+RC(q^\ast,u^\ast+d^\delta_\alpha+p)-RC(q^\ast,u^\ast)+RA(u^\ast+d^\delta_\alpha+p)e+\mu RM\frac{r^\delta_\alpha}{\|r^\delta_\alpha\|_\Vtil}=0 \label{eq:err_smooth_r} \\
\fl p_t+PC(q^\ast,u^\ast+d^\delta_\alpha +p)-PC(q^\ast,u^\ast)+PA(u^\ast+d^\delta_\alpha+p)e+\nu PNP\hat{u}=0 \nonumber \\
\fl (e,r^\delta_\alpha,p)(0)=(\hat{q}_0-q^\ast,R\hat{u}(0)-G_\alpha z^\delta(0),P(\hat{u}_0-u_0))\,, \nonumber
\end{eqnarray}
that upon the replacements $r \mapsto r^\delta_\alpha$, $u^\ast \mapsto u^\ast+d^\delta_\alpha$ and up to the perturbation $\tilde{d}^\delta_\alpha$ in (\ref{eq:err_smooth_r}) are the same as (\ref{eq:errorsys_e}) - (\ref{eq:errorsys_p}).

Here and below we set $\sigma\equiv0$. Since now we only deal with $r^\delta_\alpha$ (and not with $r$, $r^\delta_\alpha$ simultaneously as in the previous section) proofs become much more analogous to the exact data case. 
\subsubsection{Well-definedness}
For proving the well-definedness, again some lower bounds on $\mu$ and $\nu$ are required.
\begin{assumption}\label{ass:mu_nu_smooth} For all $t>0$ 
\begin{enumerate}
\item $$\fl \mu(t) \geq \frac{2}{c_M}(L_C(\|d^\delta_\alpha(t)\|_\Vtil+\|p(t)\|_\Vhat)+\|\tilde{d}^\delta_\alpha(t)\|_X)\frac{\|r^\delta_\alpha(t)\|_X \|r^\delta_\alpha(t)\|_\Vtil}{\|r^\delta_\alpha(t)\|^2_\VXtil};$$ 
\item \begin{eqnarray*}
\fl \nu(t)\geq \max \Big\{\underline{\nu}, \Big( \frac{4(L_C+C_A\|e(t)\|_Q)}{c_N}(\|p(t)\|_\Vhat+\|d^\delta_\alpha\|_\Vtil)\\
+\frac{2 C_A C_{\Vhat \VXhat}C_{\VXhat X}}{c_N}\|p(t)\|_\Vhat\Big)\frac{\|p(t)\|_X}{\|p(t)|^2_\VXhat}\Big\}.
\end{eqnarray*} 
\end{enumerate}
\end{assumption}
Therewith we get a similar well-posedness result as in the previous section.
The critical difference to Section 3.2 is that we get existence for all times.  
\begin{proposition}\label{prop:wellposed_noisesmooth}
Let Assumptions \ref{ass:equation} with (\ref{AbdVV}) and \ref{ass:mu_nu_smooth} hold, and let
$(\hat{q}_0-q^*,\hat{u}_0-u_0)\in Q\times(\Vtil+\Vhat)$.
Then there exists a solution $(\hat{q}(t),\hat{u}(t))\in Q\times(\Vtil+\Vhat)$ of (\ref{method_smooth_q}), (\ref{method_smooth_u}), (\ref{method_smooth_0}) for all times $t>0$ and satisfies the following error bounds (cf. (\ref{erp})).
\begin{enumerate}
\item For all $t>0$ :  $\|e\|_Q^2+\|r_\alpha^\delta\|_X^2\leq\|e(0)\|_Q^2+\|r_\alpha^\delta(0)\|_X^2$
\item For all $t>0$ : $\|p(t)\|_X^2\leq \max\Big\{\|p(0)\|_X^2,
\frac{C_N^2 C_{\VXhat X}^4}{c_N^2}\sup_{t>0}\|Pu^\ast(t)\|_\Vhat^2$\\
\hspace*{5cm}$+\frac{C_A C_{\VXhat X}^2}{\underline{\nu}c_N^2}(\|e(0)\|_Q^2+\|r_\alpha^\delta(0)\|_X^2)
\Big\}$;
\item  $\int_0^\infty \|p(t)\|_{\Vhat} \|r_\alpha^\delta(t)\|_X \, dt \leq \frac{\|e(0)\|_Q^2+\|r_\alpha^\delta(0)\|_X^2}{2L_C}<\infty$.
\end{enumerate}
\end{proposition}
\subsubsection{State convergence}
To obtain state convergence we have to replace the parameter choice from the exact data case in a straightforward manner
with replacements $\|p\|_\Vhat \mapsto \|p\|_\Vhat+\|d^\delta_\alpha\|_\Vtil +\|\tilde{d}^\delta_\alpha\|_\Vtil+\frac{\|\tilde{d}^\delta_\alpha\|_X}{2L_C}$ and $r \mapsto r^\delta_\alpha$. In case of smoothed noisy data we do not need conditions on $\|d^\delta_\alpha\|_\Vtil$ and therefore we can prove state convergence to $0$ as $t \rightarrow \infty$.
\begin{assumption}\label{cond:mu_smooth_2}
 There exists a constant $c_1>0$ such that for all $t>0$
$$
\mu(t)\geq \max\left\{
\frac{2}{c_M} (L_C(\|d_\alpha^\delta\|_{\Vtil}+\|p\|_{\Vhat})+\|\tilde{d}^\delta_\alpha\|_X)\,, \ c_1 \|r_\alpha^\delta(t)\|_X\right\}
\frac{\|r_\alpha^\delta(t)\|_X\|r_\alpha^\delta(t)\|_{\Vtil} }{\|r_\alpha^\delta(t)\|_{\VXtil}^2}. 
$$
\end{assumption}
\begin{theorem}[State convergence]
\label{th:state_noisesmooth}
Under Assumptions \ref{ass:equation} with (\ref{AbdVV}), \ref{ass:mu_nu_smooth} and \ref{cond:mu_smooth_2}
we have that $\left\|R\hat{u}(t)-G_\alpha z^\delta(t)\right\|_X=\|r^\delta_\alpha\|_X \rightarrow 0$ as $t \rightarrow \infty$.
\end{theorem}
%
\subsubsection{Parameter convergence}
Due to the inhomogeneity $\tilde{d}_\alpha^\delta$ in the right hand side of the equation for the ``observed'' state error (\ref{eq:err_smooth_r}) the crucial estimate for parameter convergence becomes
\begin{lemma}\label{lem:est_state_noisesmooth}
Under Assumption \ref{ass:equation} with (\ref{AbdVV}), the projected state errors $r_\alpha^\delta=R\hat{u}-G_\alpha z^\delta$ and $p$ satisfy the following relation for all $0<t_a \leq t_b \leq t_c:$
\begin{eqnarray*} 
\fl \|r_\alpha^\delta(t_c)\|_X \geq
\|\int^{t_c}_{t_b} RA(u^\ast(\tau)+d^\delta_\alpha(\tau)+p(\tau))e(t_a)d\tau \|_X -\|r_\alpha^\delta(t_b)\|_X 
-\int_{t_b}^{t_c}\|\tilde{d}_\alpha^\delta(\tau)\|_X\, d\tau \\
\fl-\tilde{L}_A^2\int_{t_b}^{t_c}\int_{t_a}^\tau \|r_\alpha^\delta(s)\|_X\, ds d\tau -L_C \int_{t_b}^{t_c}(\|d_\alpha^\delta(\tau)\|_\Vtil+\|p(\tau)\|_\Vhat)\, d\tau -C_M \int_{t_b}^{t_c} \mu(\tau)\, d\tau 
\end{eqnarray*}
\end{lemma}
To obtain parameter convergence we use the persistence of excitation and link conditions, Assumptions \ref{ass:persistence_noisy} and \ref{ass:link_noisy} with the only slight modification as compared to (\ref{thetatil})
$$\theta=\theta^\delta=\mu \frac{\|r_\alpha^\delta\|_{\VXtil}^2}{\|r_\alpha^\delta\|_{\Vtil} }.$$

Therewith we obtain:
\begin{lemma}
\label{lemma1_noisesmooth}
Let Assumptions \ref{ass:equation}, \ref{ass:persistence_noisy}, \ref{ass:link_noisy}, \ref{ass:mu_nu_smooth}, and \ref{cond:mu_smooth_2} hold.\\
Then, for any given $\gamma>\frac{2\gamma_0}{\varepsilon_0}\sup_{t>0}\|\tilde{d}_\alpha^\delta(t)\|_X$, there are $\varepsilon>0$, $T>0$ and $T_1>0$ such that for all $t_1\geq T_1$ the following holds true:\\
If the parameter error $\|e(t_1)\|_Q \geq \gamma$, then there exists a $t_2\in [t_1,t_1+T]$ such that the state error $\|r^\delta_\alpha(t_2)\|_X\geq \varepsilon$.
\end{lemma}
%
This allows us to conclude:
\begin{theorem}[Parameter convergence]
\label{th:parameter_noisesmooth}
Under Assumptions \ref{ass:equation}, \ref{ass:persistence_noisy}, \ref{ass:link_noisy}, \ref{ass:mu_nu_smooth}, and \ref{cond:mu_smooth_2}
we have that 
$$
\limsup_{t\to\infty}\left\|\hat{q}(t)-q^*\right\|_Q\leq \frac{2\gamma_0}{\varepsilon_0}\sup_{t>0}\|\tilde{d}_\alpha^\delta(t)\|_X\,.
$$
\end{theorem}
\begin{proof}
By contraposition in Lemma \ref{lemma1_noisesmooth}, for all $\gamma>\frac{2 \gamma_0}{\varepsilon_0}\sup_{t>0}\|\tilde{d}^\delta_\alpha(t)\|_X$ there exists $\varepsilon >0$, $T>0$, $T_1>0$ such that for all $t_1\geq T_1$ and for all $t_2\in[t_1,t_1+T]$: $\|r^\delta_\alpha(t_2)\|_X < \varepsilon$ implies $\|e(t)\|_Q \leq < \gamma$.

So for given $\gamma >\frac{2\gamma_0}{\varepsilon} \sup_{t>0}\|\tilde{d}^\delta_\alpha(t)\|_X$ we choose $\varepsilon$ and  $T_1>0$ according to Lemma \ref{lemma1_noisesmooth}. Then from Theorem \ref{th:state_noisesmooth} it follows that there exists $T_2 \geq T_1$ such that for all $t \geq T_2$ we have $\|r^\delta_\alpha(t_2)\|_X <\varepsilon$, hence by the above $\|e(t_1)\|_Q < \gamma$ for all $t_1 \geq T_1$. Since $\gamma$ can be chosen arbitrarily close to $\frac{2 \gamma_0}{\varepsilon_0}\sup_{t>0}\|\tilde{d}^\delta_\alpha(t)\|_X$ the assertion follows.

%
\end{proof}
\section{Numerical experiments}
\subsection{Identification of a coefficient in a degenerate diffusion equation}
Consider the problem of identifying $q=q(x)$ on a domain $\Omega\subseteq\R^d$ in the (possibly degenerate) parabolic initial boundary value problem
\begin{eqnarray}\label{eq:PDEdiff}
u_t(t,x)-\nabla(D(x)\nabla u(t,x))+q(x)u(t,x)=f(t,x)\mbox{ in }\Omega\\
u(t,x)=g(t,x)\mbox{ on }\partial \Omega \nonumber\\
u(0,x)=u_0(x) \nonumber
\end{eqnarray}
from measurements of the state  $u$ on a subdomain $\omega\subseteq \Omega$
\begin{equation}\label{eq:obsdiff}
z(t,x)=Gu(t,x)=u(t,x)\vert_\omega \nonumber
\end{equation}
Here $f(t)\in L^2(\Omega)$, $g(t)\in H^{\frac12}(\partial\Omega)$, $D\in L^\infty(\Omega)$, are assumed to be known and chosen such that for $q=q^*$ a solution $u(t)=u^*(t)\in H^2(\Omega)$ to (\ref{eq:PDEdiff}) exists for all times $t>0$. 
With the spaces 
$$
Q=H^s(\Omega)\,, \quad X=L^2(\Omega)\,, \quad Z=L^2(\omega)\,,
$$
where $s>\frac{d}{2}$ so that $Q$ is continuously embedded in $L^\infty(\Omega)$,
and the operators defined by 
$$
C(q,u)=B(u)+A(u)q \,, \quad (B(u),v)_X=\int_\Omega (D\nabla u)^T\nabla v\, dx \,, \quad A(u)q=qu \,, \quad G u=u\vert_\omega\, ,
$$
this fits into the framework of the previous sections with an appropriate choice of the spaces $\Vtil, \Vhat,\VXtil,\VXhat$ and the operators $M, N$, see below.
Note that this formulation corresponds to the standard semigroup formulation for parabolic problems in case $D>0$ (see, e.g., \cite{Evans98}). However we do not assume $D$ to be positive, not even nonnegative, hence the monotonicity assumption from \cite{Kuegler10} fails even if we use the setting there with the problem adapted spaces $V=\{v\in L^2(\Omega)\, | \, \sqD\nabla v\in L^2(\Omega)\}$ with the norm $\|v\|_V=\Bigl(\|\sqrt{|D|}\nabla v\|_{L^2(\Omega)}^2+\|v\|_{L^2(\Omega)}^2\Bigr)^{\frac{1}{2}}$, $H=L^2(\Omega)$ (with the notation $V$ and $H$ from \cite{Kuegler10}). 
The case $D<0$, often denoted as antidiffusion, for example occurs in certain models of pattern formation, see e.g. \cite{antidiffusion}.

We first of all define the spaces $\Vtil,\Vhat$ such that the Lipschitz condition on $C$ from Assumption \ref{ass:equation} holds:
\begin{eqnarray*}
\fl \Vtil&=\{v\in L^2(\Omega)\, | \, \overline{\mbox{supp}(D\nabla v)}\subseteq \omega \,, \mbox{ supp }v \subseteq \omega \mbox{ and }
\nabla(D\nabla v\vert_\omega)\in L^2(\omega)\}\\
\fl \Vhat&=\{v\in L^2(\Omega)\, | \, \overline{\mbox{supp}(D\nabla v)}\subseteq \Omega\setminus\omega \,, \mbox{ supp }v\subseteq \Omega\setminus\omega \mbox{ and }
\nabla(D\nabla v\vert_\cmega)\in L^2(\Omega\setminus\omega)\}
\end{eqnarray*}
with norms
\begin{eqnarray*}
\fl \|v\|_\Vtil= \|\nabla(D\nabla v\vert_\omega)\|_{L^2(\omega)}+\|v\|_{L^2(\Omega)}\,, \quad
\|v\|_\Vhat= \|\nabla(D\nabla v\vert_\cmega)\|_{L^2(\cmega)}+\|v\|_{L^2(\Omega)}\,.
\end{eqnarray*}
and their smooth counterparts
\begin{eqnarray*}
\fl \Vtil^\infty=\{\phi\in C_0^\infty(\Omega)\, | \, \overline{\mbox{supp}(\phi)}\mbox{ is a compact subset of }\omega\}\\
\fl \Vhat^\infty=\{\psi\in C_0^\infty(\Omega)\, | \, \overline{\mbox{supp}(\psi)}\mbox{ is a compact subset of }\cmega\}
\end{eqnarray*}
which are dense in $\Vtil$ and $\Vhat$, and whose sum $\Vtil+\Vhat$ is dense in $L^2(\Omega)$. 
Therewith the Lipschitz condition on $C(q^*,\cdot)$ is obtained as follows. For any $\phi+\psi\in\Vtil+\Vhat$ we have 
\begin{eqnarray*}
\fl (C(q^*,u^*(t)+v+w)-C(q^*,u^*(t)),\phi+\psi)_X\\
\fl =\int_\Omega \Bigl((D\nabla (v+w))^T\nabla (\phi+\psi)+q^*(v+w)(\phi+\psi)\Bigr)\, dx \\
\fl =\int_\Omega q^*(v+w)(\phi+\psi)\, dx 
+\int_\omega (D\nabla v\vert_\omega)^T\nabla \phi\vert_\omega\, dx 
+\int_\cmega (D\nabla w\vert_\cmega)^T\nabla \psi\vert_\cmega\, dx \\
\fl =\int_\Omega q^*(v+w)(\phi+\psi)\, dx 
-\int_\omega \nabla(D\nabla v\vert_\omega)\phi\vert_\omega\, dx 
-\int_\cmega \nabla(D\nabla w\vert_\cmega)\psi\vert_\cmega\, dx\\
\fl=\int_\Omega \Bigl(
-\chi_\omega[\nabla(D\nabla v\vert_\omega)]
-\chi_\cmega[\nabla(D\nabla w\vert_\cmega)]
+q^*(v+w)\Bigr)(\phi+\psi)\, dx\\
\fl\leq \max\{1,\|q^*\|_{L^\infty(\Omega)}\} (\|v\|_\Vtil+\|w\|_\Vhat) \|\phi+\psi\|_{L^2(\Omega)} \,.
\end{eqnarray*}
where $\chi_\omega:L^2(\omega)\to L^2(\Omega)$, $\chi_\cmega:L^2(\cmega)\to L^2(\Omega)$ denote the operators defined by the respective extension by zero to all of $\Omega$.
 
The operator $A(u)$ can be estimated as follows:
For all $v \in X$ we get
\begin{eqnarray*}
\fl \|A(u^\ast+v)\|_{Q\rightarrow X} =\|(u^\ast+v)q\|_{Q\rightarrow X} =\sup_{q\in Q, q\neq 0}\frac{\|(u^\ast+v)q)\|_X}{\|q\|_Q}\\
\fl \leq \sup_{q\in Q, q\neq 0}\frac{\|u^\ast+v\|_{L^2(\Omega)}\|q\|_{L^\infty{\Omega}}}{\|q\|_{H^s(\Omega)}} \leq \sup_{q\in Q, q \neq 0 }C_{H^s \rightarrow L^\infty}\frac{\|u^\ast+v\|_{L^2(\Omega)}\|q\|_{H^s(\Omega)}}{\|q\|_{H^s(\Omega)}} \\
\fl \leq C_{H^s \rightarrow L^\infty}(\|u^\ast\|_{L^2(\Omega)}+\|v\|_{L^2(\Omega)}),
\end{eqnarray*}
i.e. (\ref{AbdX}) in Assumption is satisfied with $C_A=C_{H^s \rightarrow L^\infty}\max \{1,\sup_{t>0}\|u^\ast(t)\|_L^2(\Omega)\}$ which by continuity of the embeddings $\Vtil \hookrightarrow X$ and $\Vtil+\Vhat \hookrightarrow X$ implies (\ref{AbdV}) and (\ref{AbdVV}).

The nullspace of $G$ and its orthogonal complement are given by 
\begin{eqnarray*}
\fl \cN(G)=\{w\in L^2(\Omega)\, | \, w\vert_\omega=0\}=\{w\in L^2(\Omega)\, | \, \mbox{supp}(w)\subseteq\cmega\}\,, \\
\fl \cN(G)^\bot =\{v\in L^2(\Omega)\, | \, v\vert_\cmega=0\}=\{v\in L^2(\Omega)\, | \, \mbox{supp}(v)\subseteq\omega\}\,, 
\end{eqnarray*}
and the respective projections are defined by 
$$
Ru=\chi_\omega[u\vert_\omega]\,, \quad
Pu= u-Ru=\chi_\cmega[u\vert_\cmega]\,. 
$$
We define the operators $M,N$ and the spaces $\VXtil,\VXhat$ as follows.
\begin{eqnarray}
(Mv,\phi)_X&=&\int_\omega (|D|\nabla v\vert_\omega)^T\nabla \phi\vert_\omega \, dx +\int_\Omega v\phi\, dx \quad \forall \, v\in \Vtil, \phi\in\Vtil^\infty
\label{MN1}\\
(Nw,\psi)_X&=&\int_\cmega (|D|\nabla w\vert_\cmega)^T\nabla \psi\vert_\cmega \, dx +\int_\Omega w\psi\, dx \quad \forall \, w\in \Vhat, \psi\in\Vhat^\infty
\label{MN2}
\end{eqnarray}
(making use of the fact that the spaces $\Vtil^\infty,\Vtil^\infty$ are dense in $\Vtil,\Vtil$, respectively).
Hence assuming that $D$ does not change its sign on $\omega$ and on $\cmega$ 
$(D\geq0 \mbox{ a.e. on }\omega \mbox{ or }D\leq0 \mbox{ a.e. on }\omega )$ and 
$(D\geq0 \mbox{ a.e. on }\cmega \mbox{ or }D\leq0 \mbox{ a.e. on }\cmega )$
we get
\begin{eqnarray*}
\fl \|Mv\|_X=\sup_{\phi\in\Vtil^\infty, \,\phi\not=0}\frac{(Mv,\phi)_X}{\|\phi\|_{L^2(\Omega)}}\leq \|v\|_\Vtil
\quad \forall \, v\in \Vtil\\
\fl \|Nw\|_X=\sup_{\psi\in\Vhat^\infty, \,\psi\not=0}\frac{(Mw,\psi)_X}{\|\psi\|_{L^2(\Omega)}}\leq \|w\|_\Vhat
\quad \forall \, w\in \Vhat
\end{eqnarray*}
and 
\begin{eqnarray*}
\fl (Mv,v)_X=\|\sqD\nabla v|_\omega\|_{L^2(\omega)}^2 +\|v\|_{L^2(\Omega)}^2=:\|v\|_{\VXtil}^2
\quad \forall \, v\in \Vtil\\
\fl (Nw,w)_X=\|\sqD\nabla w|_{\Omega \setminus \omega}\|_{L^2(\cmega)}^2 +\|w\|_{L^2(\Omega)}^2=:\|w\|_{\VXhat}^2
\quad \forall \, w\in \Vhat
\end{eqnarray*}
Since in this case $G$ has closed range we need not regularize in case of noisy data, i.e., we can set $\alpha=0$:
$$
u_\alpha^\delta=G^\dagger \zdel=\chi_\omega[\zdel]\,, \quad 
d_\alpha^\delta=G_\alpha(\zdel-z)=\chi_\omega[\zdel-z]=\chi_\omega[\zdel-u\vert_\omega]\,.
$$

In our implementation we consider the one dimensional case with 
domain $\Omega=(0,1)$.
The right hand side $f$ is given by $f(t,x)=\frac{1}{1+t}(D\pi^2-\frac{1}{1+t}+q^\ast(x))\sin(\pi x)$, where the exact parameter $q^\ast$ is a quadratic polynomial, $q^\ast=0.025x^2-0.025x$.
For simplicity the diffusion coefficient is chosen to be constant, $D=1$.

For solving the partial differential equation system (\ref{eq:est_exact_q})-(\ref{eq:est_exact0}) we derive its variational formulation and discretize the spaces $Q$ and $X$ by cubic Hermite basis functions $\phi_j$ and $\psi_j$ for $j=2,...,N-1$ on a uniform mesh $0=x_1 <x_2 < ...<x_N=1$, where $N=31$ in our case.
$$
\phi_j(x)=\cases{-2\left(\frac{x-x_{j-1}}{h}\right)^3+3\left(\frac{x-x_{j-1}}{h}\right)^2 \, \quad \mbox{ if } \,  x\in(x_{j-1},x_j) \\ 1-3\left(\frac{x-x_j}{h}\right)^2+2\left(\frac{x-x_j}{h}\right)^3 \, \quad \mbox{ if } \, x \in (x_j,x_{j+1})  \\ 0 \, \quad \mbox{ else}}
$$
$$
\psi_j(x)=\cases{h\left(\frac{x-x_{j-1}}{h}\right)^3 -h\left(\frac{x-x_{j-1}}{h}\right)^2 \, \quad \mbox{ if }x\in(x_{j-1},x_j) \\ h\left(\frac{x-x_j}{h}\right)^3-2h\left(\frac{x-x_j}{h}\right)^2+h\left(\frac{x-x_j}{h}\right) \, \quad \mbox{ if } \, x \in (x_j,x_{j+1})  \\
0 \, \quad \mbox{ else.}}
$$
The reason for using such high order spaces is the required regularity on arguments of the operators $M$, $N$ according to (\ref{MN1}), (\ref{MN2}).
After using these as ansatz and test function in the variational formulation for space discretiztion, we solve the resulting ODE System with an implicit Euler method with step size $h_t=0.6$.
The interesting cases are those with partial observations and noisy data.

In our experiments we just employ a simple heuristic choice of $\mu$ and $\nu$: In case of partial observations we used the lower bound for $\mu$, Assumption \ref{cond:mu02}, where the constant $c_1$ shows up. Therefore we solve the optimization problem $\min_{c_1}{\|R\hat{u}-Ru^\ast\|^2_X}$, where the constant $c_1$ varies in decimal steps between $0.1$ and $1000$, in order to find an appropriate $\mu$. 
For the case of noisy data the constants $\mu$ and $\nu$ vary between $0.1$ and $1000$ and we solve the minimization problem $\min_{\mu,\nu}{\|R\hat{u}-Ru^\ast\|^2_X}$.
This approach will be enhanced in future work. 

\begin{figure}\label{fig:figure1}
\includegraphics[width=\columnwidth]{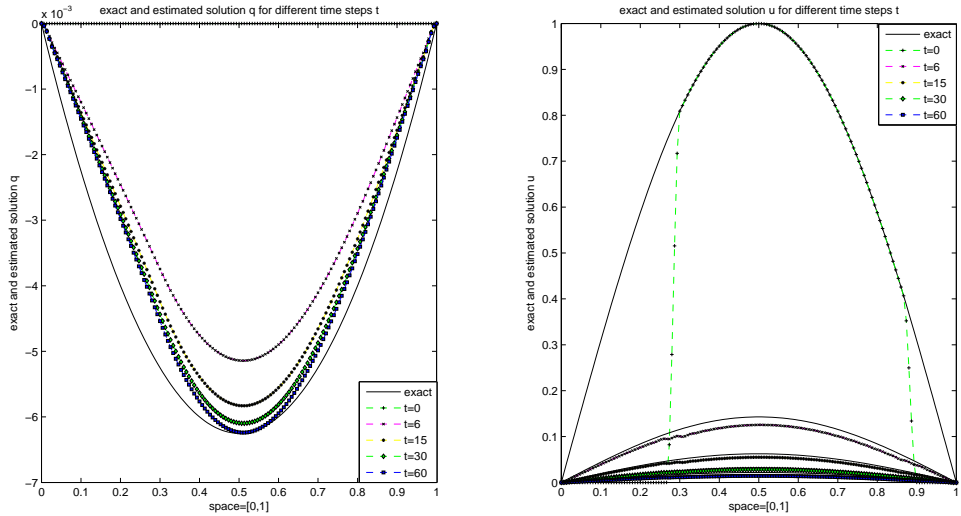}
\caption{The parameter estimate $\hat{q}(t,x)$ (left) and the state estimate $\hat{u}(t,x)$ (right) at times $t=0,6,15,30,45,60$. }
\end{figure}
To investigate the case of partial observations we restrict the data to the subinterval $\omega=(0.3,0.87)$ of $\Omega=(0,1)$.
The results for this case are shown in Figure \ref{fig:figure1}.
There on the left the estimated parameter for different times varying from $[0,60]$, are shown starting with $\hat{q}(0)=0$. The estimated parameters are the lines with markers, whereas the straight line indicates the exact parameter. 
On the right, the estimated (lines with markers) and exact (straight lines) state for different time steps are displayed. Although we have just partial observations, the state is estimated quite well also in the unobserved region $\Omega \setminus \omega$.
One can see that also the estimated parameter gets close to the exact one, but it is shifted to the right, which is due to the fact, that data are just given on the nonsymmetric interval $\omega$. Note that we do not know whether the persistence of excitation condition is satisfied here, which is in fact hard to verify in general.

\begin{figure}\label{fig:figure2}
\includegraphics[width=\columnwidth]{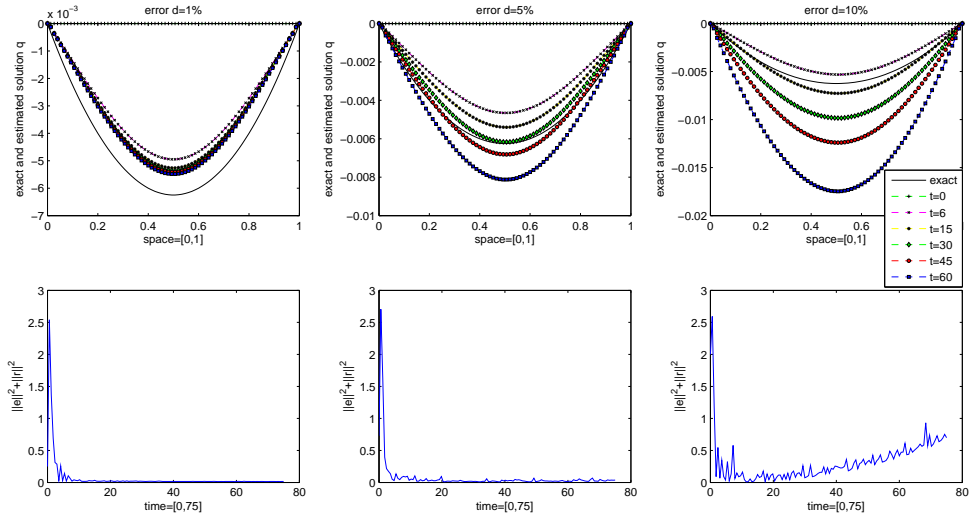}
\caption{Estimated parameter $\hat{q}(t,x)$ at different times and the observed and parameter error ($\|e\|^2_Q+\|r\|^2_X$) of Proposition \ref{well_defined_noisy} for different noise levels $\delta=1, 5, 10 \%$.}
\end{figure}																							
For the noisy data case we assumed to have data with Gaussian noise with different noise levels $\delta=1 \%, 5 \%, 10\% $. In this case of irregular noise, according to section \ref{sec:noise}, parameter and state convergence cannot be proven if $T^\ast<\infty$, so we expect closeness only for times satisfying condition (\ref{cond_d}). This can be seen in the numerical results as well, because the error is increasing from a certain time instance on, which corresponds to the semiconvergence phenomenon in regularization. As one might expect the time where the error starts to grow again gets smaller as the noise level increases.
In Figure \ref{fig:figure2} the above row shows the estimated parameter for the three different noise levels $\delta=1 \%, 5 \%, 10\%$ for different times up to $t=75$.

In Figure \ref{fig:figure2} the lower row displays the errors of the estimated observed state and parameter ($\|e\|^2_Q+\|r\|^2_X$) as in Proposition \ref{well_defined_noisy} developing over time. For small noise $\delta=1 \%$ the estimated parameter gets close to the exact one, and also the error decreases, as time proceeds, whereas for larger noise $\delta=5 \%,10 \%$ the estimated parameter first gets close to the exact parameter up to a certain time instance and then it drifts away again, hence the error increases.

\section{Conclusion and Outlook}

In this paper we have developed and analyzed an online parameter identification method for time dependent problems. The main idea was to formulate an alternative dynamic update law for the state and an additional one for the parameter estimate. 
We showed that the solution of this alternative system of differential equations is well defined and that it converges to the exact parameter and state. The proofs were done for the case of exact data as well as for the case of noisy data and smooth noisy data. The main advantages of this method are, that it imposes less restrictions on the underlying model compared to existing methods and that it is also applicable in case of partial observations. 
In a numerical example we showed the performance of our online parameter identification method. Here the results could be improved by finding optimal values for $\mu$ and $\nu$.

Another future goal is to consider time dependent parameters. This would mean that the model itself contains a dynamical update law for the parameter and therefore the online parameter identification method has to be adapted, so that this is taken into account.

\bigskip

\section*{Acknowlegdment}
This work was supported by the Karl Popper Kolleg ``Modeling-Simulation-Optimization'' funded by the Alpen-Adria-Universit\"at Klagenfurt and by the Carinthian Economic Promotion Fund (KWF).

\bigskip

\bibliographystyle{siam}
\bibliography{Paper_Online1}

\end{document}